\pdfoutput=1
\RequirePackage{ifpdf}
\ifpdf 
\documentclass[pdftex]{sigma}
\else
\documentclass{sigma}
\fi

\usepackage{mathtools}
\usepackage{xspace}
\usepackage{tabu}
\usepackage{aligned-overset}

\numberwithin{equation}{section}

\newtheorem{Theorem}{Theorem}[section]
\newtheorem*{Theorem*}{Theorem}
\newtheorem{Corollary}[Theorem]{Corollary}
\newtheorem{Lemma}[Theorem]{Lemma}

 { \theoremstyle{definition}
\newtheorem{Definition}[Theorem]{Definition}

\newtheorem{Example}[Theorem]{Example}
\newtheorem{Remark}[Theorem]{Remark} }

\newcommand*{\Z}{\mathbb Z}		
\newcommand*{\R}{\mathbb R}		
\newcommand*{\C}{\mathbb C}		
\newcommand*{\tensor}{\otimes}		

\DeclarePairedDelimiter{\scal}{\langle}{\rangle}	
\DeclareMathOperator{\Span}{span}
\DeclareMathOperator{\id}{id}		
\DeclareMathOperator{\Tr}{Tr}		
\DeclareMathOperator{\Ad}{Ad}		

\DeclareMathOperator{\Aut}{Aut}
\DeclareMathOperator{\der}{der}

\DeclareMathOperator{\gau}{gau}
\DeclareMathOperator{\Gau}{Gau}

\DeclareMathOperator{\Hom}{Hom}

\newcommand*{\Star}{$^*$\nobreakdash}
\newcommand*{\acts}{\,.\,}

\newcommand*{\hilb}{\mathfrak}		
\newcommand*{\alg}{\mathcal}		
\newcommand*{\hH}{\hilb{H}}			
\newcommand*{\one}{1}				
\newcommand*{\aA}{\alg{A}}			
\newcommand*{\aB}{\alg{B}} 			
\newcommand*{\End}{\mathcal{L}} 
\newcommand*{\Cont}{C}

\DeclarePairedDelimiterX{\lprod}[2]{{\,\prescript{}{#1}{\langle}}}{\rangle}{#2}
\DeclarePairedDelimiterX{\rprod}[2]{\langle}{\rangle_{#1}}{#2}
\DeclarePairedDelimiterX{\ketbra}[2]{\lvert}{\rvert}{#1 \delimsize\rangle \delimsize\langle #2}
\DeclareMathOperator{\Irrep}{Irr} 			

\newcommand{\doubleitem}{%
 \begingroup
 \stepcounter{enumi}%
 \edef\tmp{\theenumi,}%
 \stepcounter{enumi}
 \edef\tmp{\endgroup\noexpand\item[\tmp\labelenumi]}%
 \tmp}

 \newcommand{\tripleitem}{%
 \begingroup \stepcounter{enumi}%
 \edef\tmp{\theenumi,}%
 \stepcounter{enumi}
 \edef\tmpt{\theenumi,}
 \stepcounter{enumi}
 \edef\tmp{\endgroup\noexpand\item[\tmp\tmpt\labelenumi]}%
 \tmp}

\begin{document}
\allowdisplaybreaks

\newcommand{\arXivNumber}{2107.04653}

\renewcommand{\PaperNumber}{015}

\FirstPageHeading

\ShortArticleName{An Atiyah Sequence for Noncommutative Principal Bundles}

\ArticleName{An Atiyah Sequence\\ for Noncommutative Principal Bundles}

\Author{Kay SCHWIEGER~$^{\rm a}$ and Stefan WAGNER~$^{\rm b}$}

\AuthorNameForHeading{K.~Schwieger and S.~Wagner}

\Address{$^{\rm a)}$~iteratec GmbH, Zettachring 6, 70567 Stuttgart, Germany}
\EmailD{\href{mailto:kay.schwieger@gmail.com}{kay.schwieger@gmail.com}}
\URLaddressD{\url{https://www.xing.com/profile/Kay_Schwieger/cv}}

\Address{$^{\rm b)}$~Blekinge Tekniska H\"ogskola, SE-371 79 Karlskrona, Sweden}
\EmailD{\href{mailto:stefan.wagner@bth.se}{stefan.wagner@bth.se}}
\URLaddressD{\url{https://www.bth.se/eng/staff/stefan-wagner-stw/}}

\ArticleDates{Received July 26, 2021, in final form February 21, 2022; Published online March 07, 2022}

\Abstract{We present a derivation-based Atiyah sequence for noncommutative principal bundles. Along the way we treat the problem of deciding when a given $^*$-automorphism on the quantum base space lifts to a $^*$-automorphism on the quantum total space that commutes with the underlying structure group.}

\Keywords{Atiyah sequence; noncommutative principal bundle; freeness; factor system}

\Classification{46L87; 46L85; 55R10}

\section{Introduction}\label{sec:intro}
	
Chern--Weil theory is an important tool for many disciplines such as analysis, geometry, and mathematical physics.
For instance, it provides invariants of principal bundles and vector bundles by means of connections and curvature, and thus a way to measure their non-triviality. The~Chern--Weil homomorphism of a smooth principal bundle $q\colon P \to M$ with structure group~$G$ is an algebra homomorphism from the algebra of polynomials that are invariant under the adjoint action of $G$ on its Lie algebra, into the even de Rham cohomology $H_{\text{dR}}^{2\bullet}(M,\mathbb{K})$.
This~map is achieved by evaluating an invariant polynomial on the curvature of a connection 1-form $\omega$ on~$P$.
The latter procedure involves lifting vector fields on $M$ horizontally with respect to $\omega$ to $G$-equivariant vector fields on $P$.
An important remark in this context is that connection 1-forms on $P$ are in a bijective correspondence with $C^{\infty}(M)$-linear sections of the associated Lie algebra extension
\begin{gather}
	0 \longrightarrow \mathfrak{gau}(P)\longrightarrow \mathcal{V}(P)^G\longrightarrow \mathcal{V}(M) \longrightarrow 0, \label{eq:Atiyah}
\end{gather}
which is well-known as the so-called \emph{Atiyah sequence} of the principal bundle $P$ (see, e.g.,~\cite{Ati57,Kob69II}).
About three decades after the seminal works of Chern, Weil, and Atiyah, Lecomte described in~\cite{Lec85} a cohomological construction which generalizes the classical Chern--Weil homomorphism: Lecomte's construction associates characteristic classes to each Lie algebra extension, and the classical construction of Chern and Weil arises in this context from the Atiyah sequence above.

The work presented here is an attempt towards a derivation-based Chern--Weil theory for noncommutative principal bundles.
More precisely, our main objective is to provide a derivation-based generalization of the classical Atiyah sequence in equation~\eqref{eq:Atiyah} to the setting of noncommutative principal bundles.
For this purpose, we focus on free C\Star-dynamical systems, which provide a natural framework for noncommutative principal bundles (see, e.g., \cite{BaCoHa15,Ell00,Phi87,SchWa17c} and references therein).
Their structure theory and their relation to $K$-theory (see, e.g., \cite{CoYa13a,ForRen19,SchWa17a,SchWa17b,SchWa17c} and references therein) certainly appeal to operator algebraists and functional analysts.
Additionally, noncommutative principal bundles are becoming increasingly prevalent in various applications of geometry (cf.~\cite{IoMa16,Ivan17,Meir18,SchWa20,SchWa21}) and mathematical physics (see, e.g., \cite{BaHa-etal07,CacMes19,DaSi13,DaSiZu14,EchNeOy09,HaMa10,LaSu05,Wahl10} and references therein).

For the sake of completeness, we bring to mind that the algebraic setting of Hopf-Galois extensions already comprises abstract notions of connections, curvature and characteristic classes in terms of a universal differential calculus (see, e.g.,~\cite{BeMA20,Bre04,DGH01} and references therein).
Furthermore, we recall that Neeb~\cite{Neeb08} associated Lie group extensions with projective modules that generalize the classical Atiyah sequence for vector bundles.

Finally, we wish to mention the recent theory of pseudo-Riemannian calculus introduced by Arnlind and Wilson in~\cite{Jo17}, which constitutes a derivation-based computational framework for Riemannian geometry over noncommutative algebras (see also~\cite{Jo21, Jo18}).
It is our hope that this work will advance to the development of pseudo-Riemannian calculus.

\subsection*{Organization of the article}

Let $(\aA,G,\alpha)$ be a free C\Star-dynamical system with fixed point algebra $\aB$.
After some preliminaries, we review in Section~\ref{sec:facsys} the notion of a factor system of $(\aA,G,\alpha)$, which is the key feature of our investigation.
In fact, factor systems provide us with a natural framework for doing computations and constitute invariants for $(\aA,G,\alpha)$.
In Section~\ref{sec:lift_automorphism} we utilize factor systems to give a characterization of when a \Star-automorphism on $\aB$ can be lifted to a \Star-automorphism on $\aA$ that commutes with $\alpha$ (Theorem~\ref{thm:lift_conjugated_factor_sys}).
Moreover, we use our findings to examine an ``integrated'' version of the Atiyah sequence for noncommutative principal bundles (equation~\eqref{eq:AtiyahNCgroup}).
Section~\ref{sec:cpt_abelian} is devoted to the study of the special case when $G$ is compact Abelian.
Most notably, we get a characterization in terms of second group cohomology on the dual group of $G$ with values in the unitary group of the center of $\aB$ (Theorem~\ref{thm:lift_cohomology}).
In addition, we are able to show that if $\aA$ is commutative, then every \Star-automorphism on~$\aB$ lifts to~$\aA$ provided it leaves the class of~$\aA$ invariant (Corollary~\ref{cor:commAbelian}).
In Section~\ref{sec:1-para} we make use of the results of Section~\ref{sec:lift_automorphism} to establish a~lifting result for 1-parameter groups (Theorem~\ref{thm:lift_1_param}).
This is of particular interest in Section~\ref{sec:Atiyah}, where we finally present a generalization of the classical Atiyah sequence in equation~\eqref{eq:Atiyah} to the setting of free C\Star-dynamical systems.
Last but not least, we discuss infinitesimal objects such as connections and curvature (Section~\ref{sec:infobj}).

Finally, we would like to mention that with little effort the arguments and the results in Sections~\ref{sec:facsys},~\ref{sec:lift_automorphism}, and~\ref{sec:1-para} extend to free actions of quantum groups (see~\cite{SchWa17c}).
	
\section{Preliminaries and notation}\label{sec:pre}

This preliminary section exhibits the most fundamental definitions and notations used in this article.

{\bf About groups.}
Let $G$ be a compact group.
We write $\Irrep(G)$ for the set of equivalence classes of irreducible representations of $G$ and denote by $1 \in \Irrep(G)$ the class of the trivial representation.

{\bf About Hilbert spaces.}
Let $G$ be a compact group.
Furthermore, let $\hH_\sigma$, $\sigma \in \Irrep(G)$, be a~family of Hilbert spaces, and for each $\sigma \in \Irrep(G)$ let $T_\sigma$ be an operators on $\hH_\sigma$.
Throughout this article, we shall freely utilize the fact that $\sigma \mapsto \hH_\sigma$ and $\sigma \mapsto T_\sigma$ can be extended to arbitrary finite-dimensional representations of $G$ by taking direct sums with respect to irreducible subrepresentations.

{\bf About C\Star-dynamical systems.}
Let $\aA$ be a unital C\Star-algebra and let $G$ be a compact group that acts on $\aA$ by \Star-automorphisms $\alpha_g\colon \aA \to \aA$, $g \in G$, such that $G \times \aA \to \aA$, $(g,x) \mapsto \alpha_g(x)$ is continuous.
Throughout this article, we call such data a \emph{C\Star-dynamical system} and denote it briefly by $(\aA,G,\alpha)$.
Moreover, we typically write $\aB := \aA^G$ for the corresponding fixed point algebra.

The conditional expectation $P_1$ onto $\aB$ allows us to define a definite right $\aB$\nobreakdash-valued inner product on $\aA$ by
\begin{gather*}
	\scal{x,y}_\aB := P_1(x^*y) = \int_G \alpha_g(x^*y) \, {\rm d}g,
	\qquad
	x,y \in \aA.
\end{gather*}
The completion of $\aA$ with respect to the induced norm yields a right Hilbert $\aB$\nobreakdash-module, which we denote by $L^2(\aA)$.
The algebra $\aA$ admits a faithful \Star-representation on $L^2(\aA)$ by adjointable operators given by $\lambda\colon \aA \to \End \big( L^2(\aA) \big)$, $\lambda(x)y := x \cdot y$, and consequently we may identify $\aA$ with $\lambda(\aA) \subseteq \End \big( L^2(\aA) \big)$.
Furthermore, for each $g \in G$ we have a unitary operator $U_g$ on $L^2(\aA)$ defined for $x \in \aA \subseteq L^2(\aA)$ by $U_g x := \alpha_g(x)$.
The map $G \ni g \mapsto U_g \in \mathcal{U}\big( L^2(\aA) \big)$ is strongly continuous and implements the \Star-automorphisms $\alpha_g$, $g \in G$, via $\lambda ( \alpha_g(x)) = U_g \lambda(x) U_g^*$ for all $x \in \aA$.
Like every representation of~$G$, the algebra $\aA$ can be decomposed into its isotypic components, let us say, $A(\sigma)$, $\sigma \in \Irrep(G)$, which amounts to saying that their algebraic sum
\begin{gather*}
	\aA_f:= \bigoplus_{\sigma \in \Irrep(G)}^{\rm alg} A(\sigma)
\end{gather*}
is a dense \Star-subalgebra of~$\aA$.
Furthermore, the isotypic components are pairwise orthogonal, right Hilbert $\aB$-submodules of $L^2(\aA)$ such that $L^2(\aA) = \overline{\bigoplus}_{\pi \in \hat G} A(\pi)$.

{\bf About freeness.}
A C\Star-dynamical system $(\aA,G,\alpha)$ is called \emph{free} if the \emph{Ellwood map}
\begin{align*}
	\Phi\colon \ \aA \tensor_{\text{alg}} \aA \rightarrow C(G,\aA),
	\qquad
	\Phi(x\tensor y)(g):=x \alpha_g(y)
\end{align*}
has dense range with respect to the canonical C\Star-norm on $\Cont(G,\aA)$.
This condition was originally introduced for actions of quantum groups on C\Star-algebras by Ellwood~\cite{Ell00} and is known to be equivalent to Rieffel's saturatedness~\cite{Rieffel91} and the Peter--Weyl--Galois condition~\cite{BaCoHa15}.

One of the key tools used in this article is a characterization of freeness that we provided in~\cite[Lemma~3.3]{SchWa17c}, namely that a C\Star-dynamical system $(\aA,G,\alpha)$ is free if and only if for each irreducible representation $(\sigma,V_\sigma)$ of $G$ there is a finite-dimensional Hilbert space~$\hH_\sigma$ and an isometry $s(\sigma) \in \aA \tensor \End(V_\sigma,\hH_\sigma)$ satisfying $\alpha_g(s(\sigma))=s(\sigma) \cdot \sigma_g$ for all $g \in G$.
A rich class of free actions is given by so-called \emph{cleft} actions, which are characterized as follows:
For each irreducible representation $(\sigma,V_\sigma)$ of $G$ there is a unitary element \mbox{$u(\sigma) \in \aA \tensor \End(V_\sigma)$} such that $\alpha_g(u(\sigma))=u(\sigma) \cdot \sigma_g$ for all $g \in G$ (cf.~\cite[Definition~4.1]{SchWa17b}).

{\bf About 1-parameter groups.}
Let $\aA$ be a unital C\Star-algebra and let $(\varphi_t)_{t \in \R}$ be a 1\nobreakdash-parameter group of \Star-automorphisms $\varphi_t \in \Aut(\aA)$.
We typically use the letter $\aA^\infty$ to denote the smooth domain of $(\varphi_t)_{t \in \R}$, which is the set of elements $x \in \aA$ such that $\R \ni t \mapsto \varphi_t(x) \in \aA$ is smooth.
Moreover, we let
\begin{gather*}
	D \varphi_t \colon \ \aA^\infty \to \aA^\infty,
	\qquad
	D \varphi_t(x) := \lim_{t \to 0} \frac{\varphi_t(x)-x}{t}
\end{gather*}
stand for the corresponding \Star-derivation.

{\bf About derivations.}
Let $A$ be a unital \Star-algebra.
We let $\der(A)$ stand for the Lie algebra of \Star-derivations of $A$.
Furthermore, we write $A^\text{skew} \subseteq A$ for the subset of skew-adjoint elements, i.e., $A^{\text{skew}} := \{a \in A\colon a^* = -a \}$, and recall that each $a \in A^\text{skew}$ gives rise to a \Star-derivation defined by $A \ni x \mapsto [a,x] = ax -xa \in A$.

\section{Factor systems}\label{sec:facsys}

Let $(\aA,G,\alpha)$ be a free C\Star-dynamical system.
Furthermore, for each $\sigma \in \Irrep(G)$ let $\hH_\sigma$ be a finite-dimensional Hilbert space and let $s(\sigma) \in \aA \tensor \End(V_\sigma,\hH_\sigma)$ an isometry satisfying $\alpha_g\bigl(s(\sigma)\bigr)=s(\sigma) \cdot \sigma_g$ for all $g \in G$ (cf.~\cite[Lemma~3.3]{SchWa17c}).
For $1 \in \Irrep(G)$ we choose $\hH_1 := \C$ and let $s(1) := \one_\aA$.
A~key feature of $(\aA,G,\alpha)$ is the factor system associated with the isometries $s(\sigma)$, $\sigma \in \Irrep(G)$, (see~\cite[Definition~4.1]{SchWa17c}), which we now recall for the convenience of the reader.
First, we put $\aB := \aA^G$.
Second, for expediency, we naturally extend $\sigma \mapsto \hH_\sigma$ and $\sigma \mapsto s(\sigma)$ to arbitrary finite-dimensional representations $\sigma$ of $G$ by taking the direct sum with respect to irreducible subrepresentations.
For each finite-dimensional representation $\sigma$ of $G$ we may then define the \Star-homomorphism
\begin{gather*}
	\gamma_\sigma\colon \ \aB \to \aB \otimes \End(\hH_\sigma),
	\qquad
	\gamma_\sigma(b) := s(\sigma) (b \tensor \one_{V_\sigma}) s(\sigma)^*,
\end{gather*}
and for each pair $(\sigma,\pi)$ of finite-dimensional representations of $G$ the element
\begin{gather*}
	\omega(\sigma,\pi) := s(\sigma) s(\pi) s(\sigma \tensor \pi)^* \in \aB \tensor \End(\hH_{\sigma \tensor \pi},\hH_\sigma \tensor \hH_\pi).
\end{gather*}
Then the following relations hold:
\begin{gather}
	\label{eq:ranges_sys}
	\omega(\sigma,\pi)^* \omega(\sigma,\pi) = \gamma_{\sigma \tensor \pi}(\one_\aB),
	\qquad
	\omega(\sigma,\pi) \omega(\sigma,\pi)^* = \gamma_\sigma (\gamma_\pi(\one_\aB) ),
	\\
	\label{eq:coaction_sys}
	\omega(\sigma, \pi) \gamma_{\sigma \tensor \pi}(b) = \gamma_\sigma( \gamma_\pi(b) ) \omega(\sigma, \pi),
	\\
	\label{eq:cocycle_sys}
	\omega(\sigma, \pi) \omega(\sigma \tensor \pi, \rho)
		= \gamma_\sigma ( \omega(\pi, \rho) ) \omega(\sigma, \pi \tensor \rho)
\end{gather}
for all finite-dimensional representations $\sigma$, $\pi$, $\rho$ of $G$ and $b \in \aB$ (see~\cite[Lemma~4.3]{SchWa20}).
The triple $(\hH, \gamma, \omega)$ of the above families is referred to as the \emph{factor system of $(\aA, G, \alpha)$} associated with $s(\sigma)$, $\sigma \in \Irrep(G)$, or simply as a factor system of $(\aA, G, \alpha)$ when no explicit reference to the isometries is needed.

We recall from~\cite{SchWa17c} that, for $\sigma \in \Irrep(G)$, the isotypic component~$A(\sigma)$ of $\aA$ can be written as
\begin{gather*}
	A(\sigma) = \bigl\{ \Tr ( y s(\sigma) )\colon y \in \aB \tensor \End(\hH_\sigma, V_\sigma) \bigr\}.
\end{gather*}
In fact, the map $y \mapsto \Tr ( y s(\sigma) )$ is bijective from $\aB \tensor \End(\hH_\sigma, V_\sigma) \gamma_\sigma(\one_\aB)$ to $A(\sigma)$.
The action~$\alpha$ on the isotypic component takes the form
\begin{gather}
	\alpha_g \bigl( \Tr ( ys(\sigma) ) \bigr) = \Tr\bigl( \sigma_g \cdot ys(\sigma) \bigr)\label{eq:A_0action}
\end{gather}
for all $g \in G$ and $y \in \aB \tensor \End(\hH_\sigma,V_\sigma)$.
The multiplication between isotypic components can be written as
	\begin{gather}
		\Tr \bigl( y_\sigma s(\sigma) \bigr) \cdot \Tr \bigl( y_\pi s(\pi) \bigr)
		=
		\Tr \bigl( y_\sigma \gamma_\sigma(y_\pi) \omega(\sigma, \pi) s(\sigma \tensor \pi) \bigr)\label{eq:A_0mul}
	\end{gather}
	for all $\sigma, \pi \in \Irrep(G)$, $y_\sigma \in \aB \tensor \End(\hH_\sigma, V_\sigma)$, and $y_\pi \in \aB \tensor \End(\hH_\pi, V_\pi)$.
The element on the right hand side is, in fact, a sum of elements in the isotypic components corresponding to subrepresentations of $\sigma \tensor \pi$.
Hence equation~\eqref{eq:A_0mul} uniquely determines the multiplication on the dense \Star-subalgebra~$\aA_f$ of~$\aA$.
The involution can be phrased in terms of the factor system, too (see~\cite[Section~5]{SchWa17c} or~\cite[Lemma~2.4]{Ne13}).
In order to see this, let us fix $\sigma \in \Irrep(G)$, choose an isometric intertwiner $v_0\colon \C \to V_{\bar\sigma} \tensor V_\sigma$, and write $w_0\colon \C \to \hH_{\bar\sigma \tensor \sigma}$ for the associated isometry.
Then the element $p_\sigma := (\dim \sigma)^2 w_0^{}v_0^*$ does not depend on the choice of the intertwiner.
It can be shown that there are element $\bar y_1, \dots, \bar y_n \in \aB \tensor \End(\hH_{\bar\sigma}, V_{\bar\sigma})$ and $y_1, \dots, y_n \in \aB \tensor \End(\hH_\sigma, V_\sigma)$ such that
\begin{gather}
	\label{eq:involution_decomposition}
	p_\sigma = \sum_{k=1}^n \bar y_k \gamma_{\bar\sigma}(y_k) \omega(\bar \sigma, \sigma)
\end{gather}
and that, for each $y \in \aB \tensor \End(\hH_\sigma, V_\sigma)$, the element
\begin{gather*}
	J(y) := \frac{1}{\dim \sigma}\sum_{k=1}^n \bar y_k \gamma_{\bar\sigma} \bigl( \Tr(y_k y^*) \bigr) \in \aB \tensor \End(\hH_{\bar\sigma}, V_{\bar \sigma})
\end{gather*}
does not depend on the choice of $\bar y_1, \dots, \bar y_n$ and $y_1, \dots, y_n$.
The involution on $\aA_f$ then reads as
\begin{gather*}
	\Tr \bigl( y s(\sigma) \bigr)^* = \Tr \bigl( J(y) s(\sigma)^* \bigr)
\end{gather*}
for all $y \in \aB \tensor \End(\hH_\sigma, V_\sigma)$.

Noteworthily, the notion of a factor system of $(\aA, G, \alpha)$ only depends on the fixed point algebra $\aB$ and the group $G$, which leads to the following definition:

\begin{Definition}\label{def:conjugancy}
	Let $\aB$ be a unital C\Star-algebra and let $G$ be a compact group.
	\begin{enumerate}\itemsep=0pt
	\item 		
		A \emph{factor system for $(\aB,G)$} is a triple $(\hH, \gamma, \omega)$ comprising families of Hilbert~spaces~$\hH_\sigma$, $\sigma \in \Irrep(G)$, \Star-homomorphisms $\gamma_\sigma\colon \aB \to \aB \tensor \End(\hH_\sigma)$, $\sigma \in \Irrep(G)$, and elements $\omega(\sigma, \pi)$ in $\aB \tensor \End(\hH_{\sigma \tensor \pi}, \hH_\sigma \tensor \hH_\pi)$, $\sigma, \pi\in \Irrep(G)$, satisfying equations~\eqref{eq:ranges_sys},~\eqref{eq:coaction_sys},~\eqref{eq:cocycle_sys}, as well as the normalization conditions $\hH_1 = \C$, $\gamma_1 = \id_\aB$, and $\omega(1, \sigma) = \gamma_\sigma(\one_\aB) = \omega(\sigma, 1)$ for all $\sigma \in \Irrep(G)$.
	\item
		Two factor systems $(\hH, \gamma, \omega)$ and $(\hH', \gamma', \omega')$ for $(\aB,G)$ are called \emph{conjugated} if there are partial isometries $v(\sigma) \in \aB \tensor \End(\hH_\sigma,\hH'_\sigma)$, $\sigma \in \Irrep(G)$, normalized to $v(1) = \one_\aB$, such that
		\begin{gather*}
			\Ad[v(\sigma)] \circ \gamma_\sigma = \gamma_\sigma',
			\qquad
			\Ad[v(\sigma)^*] \circ \gamma_\sigma' = \gamma_\sigma,		\\
			v(\sigma) \gamma_\sigma\bigl( v(\pi) \bigr) \omega(\sigma, \pi)
			= \omega'(\sigma, \pi) v(\sigma \tensor \pi)
		\end{gather*}
		for all $\sigma, \pi \in \Irrep(G)$.
		In this case we write $(\hH',\gamma', \omega') = v(\hH, \gamma, \omega)v^*$ or simply $(\hH,\gamma, \omega) \sim (\hH', \gamma', \omega')$ when no reference to the partial isometries is needed.
	\end{enumerate}
Note that, above, we have implicitly used functorial versions of the families $\hH_\sigma$, $\gamma_\sigma$, $v(\sigma)$, and $\omega(\sigma, \pi)$, $\sigma, \pi \in \Irrep(G)$.
\end{Definition}

By construction, each factor system of $(\aA,G,\alpha)$ is a factor system for $(\aB,G)$ and, by \cite[Lemma~4.3]{SchWa17c}, all factor systems of $(\aA, G, \alpha)$ are conjugated. In fact, given any unital C\Star\nobreakdash-algebra~$\aB$ and any compact group~$G$, we have shown in~\cite[Section~5]{SchWa17c} that the equivalence classes of free C\Star-dynamical systems $(\aA, G, \alpha)$ with fixed point algebra~$\aB$ are in 1-to-1 correspondence with conjugacy classes of factor systems of $(\aB,G)$.

\begin{Remark}\label{rem:K0_factor_sys}
For a factor system $(\hH, \gamma, \omega)$ of $(\aA,G,\alpha)$ equations~\eqref{eq:ranges_sys} and~\eqref{eq:coaction_sys} suggest to look at the $K$-theory of~$\aB$ and the induce positive group homomorphisms $K_0(\gamma_\sigma)\colon K_0(\aB) \to K_0(\aB)$ for a finite-dimensional representation~$\sigma$ of~$G$.
	Indeed, these maps only depend on the conjugacy class of the factor system and thus amount to invariants for $(\aA, G, \alpha)$.
	In fact, the mapping $\sigma \mapsto K_0(\gamma_\sigma)$ constitutes a nice functor:
	For direct sums $\sigma \oplus \pi$ of representation $\sigma$, $\pi$ of~$G$, we obviously have $K_0(\gamma_{\sigma \oplus \pi}) = K_0(\gamma_\sigma) + K_0(\gamma_\pi)$. Moreover, by equation~\eqref{eq:coaction_sys}, we have
	\begin{gather*}
		K_0(\gamma_{\sigma \tensor \pi}) = K_0(\gamma_\sigma) \circ K_0(\gamma_\pi)
	\end{gather*}
	for all finite-dimensional representations $\sigma$, $\pi$ of $G$.
\end{Remark}

\section{Lifting an automorphism}\label{sec:lift_automorphism}

Let $(\aA, G, \alpha)$ be a free C\Star-dynamical system with fixed point algebra $\aB$ and let $\beta$ be a \Star-auto\-mor\-phism of~$\aB$.
In this section we address the question whether $\beta$ can be lifted to a \Star-automorphism~$\hat \beta$ of~$\aA$ that commutes with all~ $\alpha_g$, $g \in G$.
In the affirmative case we say that \emph{$\beta$ lifts to $\aA$} and that $\hat \beta$ is a \emph{lift of~$\beta$}.

To phrase our result we note that $\Aut(\aB)$ acts on the factor systems for $(\aB,G)$.
For $\beta \in \Aut(\aB)$ and a factor system $(\hH, \gamma, \omega)$ for $(\aB,G)$ we may define a new factor system $\big(\hH, \gamma^\beta, \omega^\beta\big)$ for $(\aB,G)$ by putting
\begin{gather*}
	\gamma^\beta_\sigma
	:= \beta \circ \gamma_\sigma \circ \beta^{-1},
	\qquad
	\omega^\beta(\sigma, \pi)
	:= \beta \bigl( \omega(\sigma, \pi) \bigr)
\end{gather*}
for all $\sigma, \pi \in \Irrep(G)$.
With this we give an answer to the above lifting problem by proving the following theorem:

\begin{Theorem}\label{thm:lift_conjugated_factor_sys}
	Let $(\aA, G, \alpha)$ be a free C\Star-dynamical system with fixed point algebra $\aB$ and let~$\beta$ be a \Star-automorphism of $\aB$.
	Then the following statements are equivalent:
	\begin{enumerate}\itemsep=0pt
	\item[$(a)$] $\beta$ lifts to $\aA$.
	\item[$(b)$] For any factor system $(\hH, \gamma, \omega)$ of $(\aA, G, \alpha)$ we have $(\hH, \gamma, \omega) \sim \big(\hH, \gamma^\beta, \omega^\beta\big)$.
	\end{enumerate}
\end{Theorem}
\begin{Remark}
	By the above theorem, a necessary condition for $\beta$ to lift to $\aA$ is that $K_0\big(\gamma_\sigma^\beta\big) = K_0(\gamma_\sigma)$ for all $\sigma \in \Irrep(G)$ or, equivalently, that $K_0(\gamma_\sigma)$ commutes with $K_0(\beta)$ for all $\sigma \in \Irrep(G)$ (cf.~Remark~\ref{rem:K0_factor_sys}).
	In particular, $K_0(\beta)$ is needs to fix the characteristic classes $[\gamma_\sigma(\one_\aB)] \in K_0(\aB)$, $\sigma \in \Irrep(G)$.
\end{Remark}

We split the proof of Theorem~\ref{thm:lift_conjugated_factor_sys} into a sequence of lemmas.
For a start we fix, for each $\sigma \in \Irrep(G)$, a finite-dimensional Hilbert space $\hH_\sigma$ and an isometry $s(\sigma)\in \aA \tensor \End(V_\sigma,\hH_\sigma)$ satisfying $\alpha_g (s(\sigma) )=s(\sigma) \cdot \sigma_g$ for all $g \in G$; for $1 \in \Irrep(G)$ we take $\hH_1 := \C$ and $s(1) := \one_\aA$.
As in Section~\ref{sec:facsys}, we write $(\hH, \gamma, \omega)$ for the associated factor system.
Our first result establishes the implication ``(a)$\Rightarrow$(b)'' of Theorem~\ref{thm:lift_conjugated_factor_sys}:

\begin{Lemma}\label{lem:conjugated}
	If $\beta$ lifts to $\aA$, then $(\hH, \gamma, \omega)$ and $\big(\hH, \gamma^\beta, \omega^\beta\big)$ are conjugated.
\end{Lemma}
\begin{proof}
	Let $\hat \beta$ be a lift of $\beta$.
	For each $\sigma \in \Irrep(G)$ we put $s^\beta(\sigma) := \hat \beta ( s(\sigma) )$ and note that $s^\beta(\sigma)$ is an isometry in $\aA \tensor \End(V_\sigma,\hH_\sigma)$ satisfying $\alpha_g\big(s^\beta(\sigma)\big)=s^\beta(\sigma) \cdot \sigma_g$ for all $g \in G$.
	Clearly, $s^\beta(1) = \one_\aA$.
	Furthermore, it is readily checked that the associated factor system is equal to $\big(\hH, \gamma^\beta, \omega^\beta\big)$.
	The~claim thus follows from the fact that all factor systems of $(\aA, G, \alpha)$ are conjugated.
\end{proof}

The task is now to prove the converse implication, ``(b)$\Rightarrow$(a)'', of Theorem~\ref{thm:lift_conjugated_factor_sys}.
For this purpose, we consider partial isometries $v(\sigma) \in \aB \tensor \End(\hH_\sigma)$, $\sigma \in \Irrep(G)$, normalized to $v(1) = \one_\aB$, such that $ \big(\hH, \gamma^\beta, \omega^\beta\big) = v (\hH, \gamma, \omega) v^*$.
Then for each $\sigma \in \Irrep(G)$ we obtain a well-defined map on the isotypic component $A(\sigma)$ by putting
\begin{gather}\label{eq:beta}
	\hat\beta_\sigma \bigl( \Tr (y s(\sigma)) \bigr) := \Tr \bigl( \beta(y) v(\sigma) s(\sigma) \bigr)
\end{gather}
for all $y\in \aB \tensor \End(\hH_\sigma, V_\sigma)$.
Taking direct sums gives a map~$\hat\beta$ on the dense \Star-subalgebra~$\aA_f$.
Due to the normalizations $v(1) = \one_\aB$ and $s(1) = \one_\aA$, we have $\hat \beta(b) = \beta(b)$ for all $b \in \aB$.
That~is, $\hat \beta$ extends $\beta$.
Furthermore, a few moments of thought show that $\hat \beta$ is bijective.
In fact, its inverse is given by the direct sum of the maps $\hat\beta^{-1}\colon A(\sigma) \to A(\sigma)$, $\sigma \in \Irrep(G)$, defined by
\[
\hat\beta^{-1}_\sigma \bigl(\Tr ( y s(\sigma)) \bigr)
	= \Tr \bigl( \beta^{-1} (y v(\sigma)^*) s(\sigma) \bigr)
\] for all $y \in \aB \tensor \End(\hH_\sigma, V_\sigma)$.
We proceed to establish further properties.

\begin{Lemma}\label{lem:prop_beta}
	The following assertions hold for the map $\hat \beta$ on $\aA_f$:
	\begin{enumerate}\itemsep=0pt
	\item[$1.$]
		$\hat\beta \circ \alpha_g = \alpha_g \circ \hat \beta$ for all $g \in G$.
	\item[$2.$]
		$\hat\beta$ is multiplicative.
	\item[$3.$]
		$\hat\beta$ is involutive.
	\end{enumerate}
\end{Lemma}
\begin{proof}1.~This is immediate from equations~\eqref{eq:A_0action} and~\eqref{eq:beta}.

2.~Let $\sigma, \pi \in \Irrep(G)$, let $x_\sigma = \Tr (y_\sigma s(\sigma)) \in A(\sigma)$ for some $y_\sigma \in \aB \tensor \End(\hH_\sigma, V_\sigma)$, and let $x_\pi = \Tr ( y_\pi s(\pi) ) \in A(\pi)$ for some $y_\pi \in \aB \tensor \End(\hH_\pi, V_\pi)$.
		Then
		\begin{align*}
			\hat\beta(x_\sigma) \cdot \hat\beta(x_\pi)
			&=
			\hat\beta \bigl( \Tr( y_\sigma s(\sigma) ) \bigr)
			\cdot \hat \beta \bigl( \Tr (y_\pi s(\pi) ) \bigr)
			\\
			&=
			\Tr \bigl( \beta(y_\sigma) v(\sigma) s(\sigma) \bigr)
			\cdot \id \tensor \Tr \bigl( \beta(y_\pi) v(\pi) s(\pi) \bigr)
			\\
			&=
			\Tr \bigl( \beta(y_\sigma) v(\sigma) \gamma_\sigma ( \beta(y_\pi) v(\pi) ) \omega(\sigma, \pi) s(\sigma \tensor \pi) \bigr).
		\end{align*}
		Using the conjugacy equations in Definition~\ref{def:conjugancy} now yields
		\begin{align*}
			\hat\beta(x_\sigma) \cdot \hat\beta(x_\pi)
			&=
			\Tr \bigl( \beta(y_\sigma) \gamma^\beta_\sigma ( \beta(y_\pi) ) v(\sigma) \gamma_\sigma ( v(\pi) ) \omega(\sigma, \pi) s(\sigma \tensor \pi) \bigr)
			\\
			&= \Tr \bigl( \beta ( y_\sigma \gamma_\sigma(y_\pi) ) \omega^\beta(\sigma, \pi) v(\sigma \tensor \pi) s(\sigma \tensor \pi) \bigr)
			\\
			&= \Tr \bigl( \beta ( y_\sigma \gamma_\sigma(y_\pi) \omega(\sigma, \pi) ) v(\sigma \tensor \pi) s(\sigma \tensor \pi) \bigr)
			=
			\hat\beta(x_\sigma \cdot x_\pi),
		\end{align*}
		which establishes that $\hat\beta$ is multiplicative.

3.~Let $\sigma \in \Irrep(G)$.
		To deal with the involution, we choose $\bar y_1, \dots \bar y_n \in \aB \tensor \End(\hH_{\bar\sigma}, V_{\bar\sigma})$ and $y_1, \dots, y_n \in \aB \tensor \End(\hH_\sigma, V_\sigma)$ satisfying equation~\eqref{eq:involution_decomposition} (cf.~Section~\ref{sec:facsys}).
		Now, let $x = \Tr ( y s(\sigma) )$ be in $A(\sigma)$ for some $y \in \aB \tensor \End(\hH_\sigma, V_\sigma)$.
		Then
		\begin{align*}
			\hat\beta(x^*)
			&= \hat\beta \bigl( \Tr ( J(y) s(\bar \sigma) ) \bigr)
			= \frac{1}{\dim \sigma} \sum_{k=1}^n \hat\beta \bigl( \Tr ( \bar y_k \gamma_{\bar\sigma} (\Tr(y_k y^*) ) s(\bar \sigma) ) \bigr)
			\\
			&=
			\frac{1}{\dim \sigma} \sum_{k=1}^n \Tr \bigl( \beta ( \bar y_k \gamma_{\bar \sigma} ( \Tr(y_k y^*) ) ) v(\bar \sigma) s(\bar\sigma) \bigr).
		\end{align*}
		It is easily checked that $\bar y_k^\beta := \beta(\bar y_k) v(\bar\sigma)$ and $y_k^\beta := \beta(y_k)v(\sigma)$, $1 \le k \le n$, also satisfy equation~\eqref{eq:involution_decomposition}.
		Since there is no loss of generality in assuming $y = y \gamma_\sigma(\one_\aB)$, we thus get
		\begin{align*}
			\hat\beta(x)^*
			&= \Tr ( \beta(y) v(\sigma) s(\sigma) )^*
			= \Tr \bigl( J ( \beta(y) v(\sigma) ) s(\bar\sigma) \bigr)
			\\
			&=
			\frac{1}{\dim \sigma} \sum_{k=1}^n \Tr \bigl( \beta(\bar y_k) v(\bar\sigma) \gamma_{\bar\sigma} ( \Tr ( \beta(y_k) v(\sigma) v(\sigma)^* \beta(y^*) ) ) s(\bar \sigma) \bigr)
			\\
			&=
			\frac{1}{\dim \sigma} \sum_{k=1}^n \Tr \bigl( \beta(\bar y_k) v(\bar\sigma) \gamma_{\bar\sigma} ( \beta \tensor \Tr( y_k y^*) ) s(\bar \sigma) \bigr).
		\end{align*}
		Invoking the conjugacy equations in Definition~\ref{def:conjugancy} finally gives
		\begin{align*}
			\hat\beta(x)^*
			&=
			\frac{1}{\dim \sigma} \sum_{k=1}^n \Tr \bigl( \beta(\bar y_k) \gamma^\beta_{\bar\sigma} ( \beta \tensor \Tr( y_k y^*) ) v(\bar\sigma) s(\bar \sigma) \bigr)
			\\
			&=
			\frac{1}{\dim \sigma} \sum_{k=1}^n \Tr \bigl( \beta ( \bar y_k) \gamma_{\bar\sigma}( \Tr( y_k y^*) ) v(\bar\sigma) s(\bar \sigma)\bigr)
			= \hat\beta(x^*),
		\end{align*}
		which shows that $\hat \beta$ is involutive.
\end{proof}

\begin{Lemma}\label{lem:beta_iso}
	We have $\rprod{\aB}{\hat\beta(x_1), \hat\beta(x_2)} = \beta\bigl( \rprod{\aB}{x_1, x_2} \bigr)$ for all $x_1, x_2 \in \aA_f$.
\end{Lemma}
\begin{proof}
	Since isotypic components are pairwise orthogonal, it suffices to show $\rprod{\aB}{\hat\beta(x_1), \hat\beta(x_2)} \allowbreak = \beta\bigl( \rprod{\aB}{x_1, x_2} \bigr)$ for all $x_1, x_2 \in A(\sigma)$ and $\sigma \in \Irrep(G)$.
	To this end, let $\sigma \in \Irrep(G)$ and let $x_1, x_2 \in A(\sigma)$.
	By Lemma~\ref{lem:prop_beta}, we obtain
	\begin{align*}
		\rprod{\aB}{\hat\beta(x_1), \hat\beta(x_2)}
		&=
		P_1\big( \hat\beta(x_1)^*\hat\beta(x_2) \big)
		=
		P_1\big( \hat\beta(x_1^*x_2) \big).
	\end{align*}
	We now decompose $x_1^*x_2$ as $\sum_{i \in I} \Tr( y_i s(\sigma_i) )$ for some mutually distinct representations $\sigma_i \in \Irrep(G)$ and $y_i \in \aB \tensor \End(\hH_{\sigma_i},V_{\sigma_i})$ and note that there is $i_0 \in I$ such that $\sigma_{i_0} = 1$, which is due to the fact that $\sigma \tensor \bar \sigma$ contains $1$ as a subrepresentation.
	Hence
	\begin{gather*}
		P_1\big( \hat\beta(x_1^*x_2) \big)
		=
		P_1\big( \beta(y_{i_0}) \big)
		=
		\beta\big( P_1(y_{i_0}) \big)
		=
		\beta\big( P_1(x_1^*x_2) \big)
		=
		\beta\bigl( \rprod{\aB}{x_1, x_2} \bigr).
		\tag*{\qed}
	\end{gather*}
	\renewcommand{\qed}{}
\end{proof}

By Lemma~\ref{lem:beta_iso}, the bijectivity of $\hat \beta$, and the fact that $\aA_f$ is dense in $L^2(\aA)$, we may extend~$\hat \beta$ to a unitary map, let's say, $U$ on $L^2(\aA)$.
Consequently, there is a \Star-automorphism on~$\aA$, for which we shall use the same letter $\hat \beta$ by a slight abuse of notation, such that
\begin{align*}
	\lambda\big( \hat \beta(x) \big) = U \lambda(x) U^*
\end{align*}
for all $x \in \aA$.
It is easily checked that $\hat \beta$ extends $\beta$ and commutes with all $\alpha_g$, $g\in G$.
Summarizing, we have shown the implication ``(b)$\Rightarrow$(a)'' of Theorem~\ref{thm:lift_conjugated_factor_sys}, which concludes the proof of this theorem.

We now turn to an application of our findings:
The group
\begin{gather*}
	\Aut_G(\aA) := \bigl\{ \varphi \in \Aut(\aA)\colon \alpha_g \circ \varphi = \varphi\circ \alpha_g \quad \forall g \in G \bigr\}
\end{gather*}
admits a short exact sequence
\begin{gather}\label{eq:AtiyahNCgroup}
	1 \longrightarrow \Gau(\aA) \longrightarrow \Aut_G(\aA)\longrightarrow \Aut(\aB)_{[\aA]} \longrightarrow 1,
\end{gather}
where
\begin{gather*}
	\Gau(\aA) := \bigl\{\varphi \in \Aut_G(\aA)\colon \varphi_{\mid \aB} = \id_\aB \bigr\}
\end{gather*}
is the group of \emph{gauge transformations} of $(\aA, G, \alpha)$ and $\Aut(\aB)_{[\aA]} \subseteq \Aut(\aB)$ is the group of \Star-automorphisms that lift to $\aA$.
Theorem~\ref{thm:lift_conjugated_factor_sys} states that $\Aut(\aB)_{[\aA]}$ can be characterized in terms of a factor system $(\hH, \gamma, \omega)$ of $(\aA, G, \alpha)$ as
\begin{gather*}
	\Aut(\aB)_{[\aA]} = \bigl\{ \beta \in \Aut(\aB)\colon (\hH,\gamma, \omega) \sim \big(\hH, \gamma^\beta, \omega^\beta\big) \bigr\}.
\end{gather*}
Looking at equation~\eqref{eq:beta}, we easily see that different choices of $v(\sigma)$, $\sigma \in \Irrep(G)$, amount to different lifts.
Hence the construction, in fact, shows that $\Gau(\aA)$ is isomorphic to the group
\begin{gather*}
	U(\hH, \gamma, \omega) := \bigl\{ u = (u(\sigma))_{\sigma \in \Irrep(G)}\colon u(\hH, \gamma, \omega)u^* = (\hH, \gamma, \omega) \bigr\},
\end{gather*}
which consists of all families of unitaries $u(\sigma) \in \gamma_\sigma(\one_\aB) \big( \aB \tensor \End(\hH_\sigma) \big) \gamma_\sigma(\one_\aB)$, $\sigma \in \Irrep(G)$, that lie in the commutant of $\gamma_\sigma(\aB)$ and satisfy the equation
\begin{gather*}
	u(\sigma) \gamma_\sigma ( u(\pi)) \omega(\sigma, \pi) = \omega(\sigma, \pi) u(\sigma \tensor \pi)
\end{gather*}
for all $\sigma, \pi \in \Irrep(G)$.

\section{The special case of a compact Abelian structure group}\label{sec:cpt_abelian}

Let $G$ be a compact Abelian group, let $(\aA, G, \alpha)$ be a free C\Star-dynamical system with fixed point algebra $\aB$, and let $\beta$ be a \Star-automorphism of $\aB$.
In the previous section we showed in Theorem~\ref{thm:lift_conjugated_factor_sys} that $\beta$ lifts to $\aA$ if and only if $(\hH,\gamma, \omega) \sim \big(\hH, \gamma^\beta, \omega^\beta\big)$ for any factor system $(\hH, \gamma, \omega)$ of $(\aA, G, \alpha)$.
In this section we give another characterization of when $\beta$ lifts to $\aA$ in terms of second group cohomology on the dual group $\hat G := \Hom(G,\mathbb{T})$ with values in the group $\mathcal{U}Z(\aB)$ of unitaries in the center of $\aB$.

To begin with, let us fix, for each $\sigma \in \hat G$, a finite-dimensional Hilbert space $\hH_\sigma$ and an isometry $s(\sigma)\in \aA \tensor \End(\C,\hH_\sigma)$ satisfying $\alpha_g(s(\sigma)) = \sigma(g) \cdot s(\sigma)$ for all $g \in G$.
For the trivial character, denoted by $0$, we choose $\hH_0 := \C$ and $s(0) := \one_\aA$.
Just as before, we write $(\hH, \gamma, \omega)$ for the associated factor system.
We shall also make use of the so-called \emph{Fr\"ohlich map} \mbox{$\Delta\colon \hat G \to \Aut ( \mathcal{U}Z(\aB) )$} associated with $(\aA, G, \alpha)$, which is for each $\sigma \in \hat G$ defined by restricting the map
\begin{gather*}
	\Delta_\sigma\colon \ \aB \to \aB,
	\qquad \Delta_\sigma(b) := s(\sigma)^* b s(\sigma).
\end{gather*}
We point out that, since all factor systems of $(\aA, G, \alpha)$ are conjugated, the Fr\"ohlich map does not depend on the choice of the factor system.

Given a lift $\hat \beta$ of $\beta$, we have seen in the proof of Lemma~\ref{lem:conjugated} that $v(\sigma) := \hat \beta\big( s(\sigma) \big) s(\sigma)^*$, $\sigma \in \hat G$, are partial isometries in $\aB \tensor \End(\hH_\sigma)$, $\sigma \in \hat G$, normalized to $v(0)=\one_\aB$, such~that
\begin{gather}	\label{eq:comm_gamma}
	\gamma^\beta_\sigma = \Ad[v(\sigma)] \circ \gamma_\sigma,
	\qquad
	\gamma_\sigma = \Ad[v(\sigma)^*] \circ \gamma_\sigma^\beta,
	\\
	v(\sigma) \gamma_\sigma( v(\pi)) \omega(\sigma, \pi)
	= \omega^\beta(\sigma, \pi) v(\sigma + \pi)
	\label{eq:comm_omega}
\end{gather}
for all $\sigma, \pi \in \hat G$.
A moment's thought shows that equation~\eqref{eq:comm_omega} can be rewritten as
\begin{gather*}
	s(\sigma + \pi)^* \beta^{-1}\big( v(\sigma + \pi) \omega(\sigma,\pi)^* \gamma_\sigma\big(v(\pi)^*\big) v(\sigma)^*\big) s(\sigma) s(\pi) = \one_\aB.
\end{gather*}

Our objective is now to give a group cohomological interpretation of the latter equation.
For~this purpose, let us for a moment assume that for all $\sigma \in \hat G$ the \Star-homomorphisms~$\gamma_\sigma$ and~$\gamma_\sigma^\beta$ are conjugated, i.e., there is a partial isometry $v(\sigma) \in \aB \tensor \End(\hH_\sigma)$ satisfying equation~\eqref{eq:comm_gamma}.
We freely use the fact that there is no loss of generality in assuming that $\gamma_\sigma(\one_\aB)$ and $\gamma_\sigma^\beta(\one_\aB)$ are the initial and final projections of~$v(\sigma)$, respectively.
Let us now consider the map $u\colon \hat G \times \hat G \to \aB$ defined by
\begin{gather}\label{eq:unitary_central_element}
	u(\sigma, \pi)
	:=
	s(\sigma + \pi)^* \beta^{-1}\big( v(\sigma + \pi) \omega(\pi,\sigma)^* \gamma_\pi\big(v(\sigma)^*\big) v(\pi)^*\big) s(\pi) s(\sigma).
\end{gather}

\begin{Lemma}\label{lem:unitary_central_cocycle}
	The following assertions hold:
	\begin{enumerate}\itemsep=0pt
	\item[$1.$]
		$u(\sigma, \pi)$ is central in $\aB$ for all $\sigma, \pi \in \hat G$.
	\item[$2.$]
		$u(\sigma, \pi)$ is unitary for all $\sigma, \pi \in \hat G$.
	\item[$3.$]
		$u$ constitutes a $2$-cocycle, i.e., for all $\sigma, \pi, \rho \in \hat G$ it satisfies
		\begin{gather*}
			u(\sigma + \pi, \rho) u(\sigma, \pi)
			= u(\sigma, \pi + \rho) \Delta_\sigma ( u(\pi, \rho) ).
		\end{gather*}
	\end{enumerate}
\end{Lemma}

\begin{Remark}We recall from~\cite[Corollary~3.11]{SchWa17a} that, for each $\sigma \in \hat G$, the isotypic component~$A(\sigma)$ is a Morita equivalence bimodule over $\aB$ and that, for all $\sigma, \pi \in \hat G$, the canonical multiplication map $m_{\sigma, \pi}\colon A(\sigma) \tensor_\aB A(\pi) \to A(\sigma + \pi)$ is an isomorphisms of Morita equivalence bimodules over~$\aB$.
	All~assertions in Lemma~\ref{lem:unitary_central_cocycle} may be derived from the facts that $m_{\sigma + \pi, \rho} \circ (m_{\sigma, \pi} \tensor \id) = m_{\sigma, \pi + \rho} \circ (\id \tensor m_{\pi, \rho})$ for all $\sigma,\pi,\rho \in \hat G$ and that the maps
	\begin{gather*}
		\Phi(\sigma,\pi)\colon \ A(\sigma + \pi) \to A(\sigma + \pi),
		\qquad
		\Phi(\sigma, \pi) := \hat\beta_{\sigma+ \pi}^{-1} \circ m_{\sigma, \pi} \circ \big(\hat\beta_\sigma \tensor \hat\beta_\pi\big) \circ m_{\sigma, \pi}^*
	\end{gather*}
	for all $\sigma, \pi \in \hat G$, with $\hat \beta_\sigma$, $\sigma \in \hat G$, being defined by equation~\eqref{eq:beta}, are automorphisms of Morita equivalence bimodules over $\aB$.
	Instead, we decided to present basic direct proofs despite rather long computations.
\end{Remark}
\begin{proof}
1.~Let $\sigma, \pi \in \hat G$.
	Then for each $b \in \aB$ we have
	\begin{alignat*}{2}
		b u(\sigma, \pi)
		&= &&bs(\sigma + \pi)^* \beta^{-1}\big( v(\sigma + \pi) \omega(\pi,\sigma)^* \gamma_\pi (v(\sigma)^* ) v(\pi)^*\big) s(\pi) s(\sigma)
			\\
			&= &&s(\sigma + \pi)^* \gamma_{\sigma + \pi}(b) \beta^{-1}\big( v(\sigma + \pi) \omega(\pi,\sigma)^* \gamma_\pi (v(\sigma)^* ) v(\pi)^*\big) s(\pi) s(\sigma)
			\\
			\overset{\eqref{eq:comm_gamma}}&{=} \,
			&&s(\sigma + \pi)^* \beta^{-1}\big( v(\sigma + \pi) \gamma_{\sigma + \pi} ( \beta(b) ) \omega(\pi,\sigma)^* \gamma_\pi (v(\sigma)^* ) v(\pi)^*\big) s(\pi) s(\sigma)
			\\
			\overset{\eqref{eq:coaction_sys}}&{=} \,
			&&s(\sigma + \pi)^* \beta^{-1} ( v(\sigma + \pi) \omega(\pi,\sigma)^* \gamma_\pi ( \gamma_\sigma ( \beta(b) )v(\sigma)^* ) v(\pi)^* ) s(\pi) s(\sigma)
			\\
			\overset{\eqref{eq:comm_gamma}}&{=} \,
			&&s(\sigma + \pi)^* \beta^{-1}\big( v(\sigma + \pi) \omega(\pi,\sigma)^* \gamma_\pi (v(\sigma)^* ) v(\pi)^*\big) \gamma_\pi ( \gamma_\sigma(b) ) s(\pi) s(\sigma)
			\\
			&= &&s(\sigma + \pi)^* \beta^{-1}\big( v(\sigma + \pi) \omega(\pi,\sigma)^* \gamma_\pi (v(\sigma)^* ) v(\pi)^*\big) s(\pi) s(\sigma) b = u(\sigma, \pi) b.
	\end{alignat*}
	That is, $u(\sigma, \pi)$ is central in $\aB$ as claimed.

2.~Let $\sigma, \pi \in \hat G$.
	Then $u(\sigma, \pi)^* u(\sigma, \pi)$ is equal to
	\begin{gather*}		
 s(\sigma)^* s(\pi)^* \beta^{-1}\big( v(\pi) \gamma_\pi (v(\sigma) ) \omega(\pi,\sigma) v(\sigma + \pi)^*\big)
		\\		
\	\qquad\quad {}\times	\gamma_{\sigma + \pi}(\one_\aB) \cdot \beta^{-1}\big( v(\sigma + \pi) \omega(\pi,\sigma)^* \gamma_\pi (v(\sigma)^* ) v(\pi)^*\big) s(\pi) s(\sigma)
		\\
\ \qquad{} =
		 s(\sigma)^* s(\pi)^* \beta^{-1}\big( v(\pi) \gamma_\pi (v(\sigma) ) \omega(\pi,\sigma) \omega(\pi,\sigma)^* \gamma_\pi (v(\sigma)^* ) v(\pi)^* \big) s(\pi) s(\sigma)
		\\
		\qquad {} \overset{\eqref{eq:ranges_sys}} {=}
		 s(\sigma)^* s(\pi)^* \beta^{-1}\big( v(\pi) \gamma_\pi (v(\sigma) ) \gamma_\pi ( \gamma_\sigma( \one_\aB) ) \gamma_\pi (v(\sigma)^* ) v(\pi)^* \big) s(\pi) s(\sigma)			
		\\
	\qquad {}	\overset{\eqref{eq:comm_gamma}} {=}
		 s(\sigma)^* s(\pi)^* \gamma_\sigma ( \gamma_\pi(\one_\aB) ) s(\sigma) s(\pi) = \one_\aB.
	\end{gather*}
	In other words, $u(\sigma, \pi)$ is an isometry, and hence unitary due to part~1 above.

3.~Let $\sigma, \pi, \rho \in \hat G$.
	To prove that $u(\sigma + \pi, \rho) u(\sigma, \pi)
	= u(\sigma, \pi + \rho) \Delta_\sigma ( u(\pi, \rho) )$, we examine both sides of the equation.
	Indeed, the left hand side reads as
	\begin{gather*}
s(\sigma + \pi + \rho)^* \beta^{-1}\big( v(\sigma + \pi + \rho) \omega(\rho,\sigma + \pi)^* \gamma_\rho (v(\sigma + \pi)^* ) v(\rho)^* \big) s(\rho)
		\\
\ \qquad\quad {}\times \gamma_{\sigma + \pi}(\one_\aB) \cdot \beta^{-1}\big( v(\sigma + \pi) \omega(\pi,\sigma)^* \gamma_\pi (v(\sigma)^* ) v(\pi)^*\big) s(\pi) s(\sigma)
		\\
\ \qquad{} =
		 s(\sigma + \pi + \rho)^* \beta^{-1}\big( v(\sigma + \pi + \rho) \omega(\rho,\sigma + \pi)^* \gamma_\rho (v(\sigma + \pi)^* ) v(\rho)^* \big) s(\rho)
		\\
\ \qquad\quad {} \times \beta^{-1}\big( v(\sigma + \pi) \omega(\pi,\sigma)^* \gamma_\pi (v(\sigma)^* ) v(\pi)^*\big) s(\pi) s(\sigma)
		\\
\ \qquad {} = s(\sigma + \pi + \rho)^* \beta^{-1}\big( v(\sigma + \pi + \rho) \omega(\rho,\sigma + \pi)^* \gamma_\rho (v(\sigma + \pi)^* ) v(\rho)^* \big)
		\\
\ \qquad \quad {}\times \big(\gamma_\rho \circ \beta^{-1} \big) \big( v(\sigma + \pi) \omega(\pi,\sigma)^* \gamma_\pi (v(\sigma)^* ) v(\pi)^*\big) s(\rho) s(\pi) s(\sigma)
		\\
\qquad{} \overset{\eqref{eq:comm_gamma}} {=}
		 s(\sigma + \pi + \rho)^* \beta^{-1}\big( v(\sigma + \pi + \rho) \omega(\rho,\sigma + \pi)^* \gamma_\rho ( \omega(\pi,\sigma)^* )
		\\
\ \qquad \quad {} \times \gamma_\rho ( \gamma_\pi ( v(\sigma)^* ) ) \gamma_\rho ( v(\pi)^* ) v(\rho)^* \big) s(\rho) s(\pi) s(\sigma).
	\end{gather*}
	Moreover, a similar computation shows that the right hand side is given by
	\begin{gather*}
	 s(\sigma + \pi +\rho)^* \beta^{-1}\big( v(\sigma + \pi + \rho) \omega(\pi + \rho,\sigma)^* \omega(\rho,\pi)^*
		\\			
\qquad \quad{}\times \gamma_\rho ( \gamma_\pi ( v(\sigma)^* ) ) \gamma_\rho ( v(\pi)^* ) v(\rho)^* \big) s(\rho) s(\pi) s(\sigma).	
	\end{gather*}
	Comparing these two expressions, we see that the 2-cocycle equation will be established once we verify that $\omega(\rho,\sigma + \pi)^* \gamma_\rho\big( \omega(\pi,\sigma)^* \big)= \omega(\pi + \rho,\sigma)^* \omega(\rho,\pi)^*$.
	But this is immediate from equation~\eqref{eq:cocycle_sys}.
\end{proof}

\begin{Lemma}\label{lem:independency}
	The cohomology class of $u$ is independent of all choices made.
	More precisely, the following holds:
	Let $(\hH', \gamma', \omega')$ be another factor system of $(\aA, G, \alpha)$ with generators $s'(\sigma)$, $\sigma \in \hat G$.
	Furthermore, let $v'(\sigma) \in \aB \tensor \End(\hH_\sigma')$, $\sigma \in \hat G$, be partial isometries, normalized to $v'(0) = \one_\aB$, satisfying equation~\eqref{eq:comm_gamma} for $\gamma_\sigma'$.
	Write $u'\colon \hat G \times \hat G \to \mathcal{U}Z(\aB)$ for the corresponding $2$-cocycle defined by equation~\eqref{eq:unitary_central_element}.
	Then $u$ and $u'$ are cohomologous.
\end{Lemma}
\begin{proof}
	As a preliminary step, let us denote by $w(\sigma) := s(\sigma) s'(\sigma)^*$, $\sigma \in \hat G$, the family in $\aB \tensor \End(\hH'_\sigma,\hH_\sigma)$ implementing the conjugation between $(\hH', \gamma', \omega')$ and $(\hH, \gamma, \omega)$.
	We define a map $u\colon \hat G \to \aB$ by
	\begin{gather*}
		u(\sigma) := s(\sigma)^* \beta^{-1} \bigl( v(\sigma) w(\sigma) v'(\sigma)^* \bigr) s'(\sigma).
	\end{gather*}
	For each $\sigma \in \hat G$ the element $u(\sigma)$ is, in fact, central in $\aB$, because for each $b\in \aB$ we have
	\begin{align*}
		u(\sigma) b
		&=
		s(\sigma)^* \beta^{-1} \bigl( v(\sigma) w(\sigma) v'(\sigma)^* \beta\gamma_\sigma'(b) \bigr) s'(\sigma)
		\\
		\overset{\eqref{eq:comm_gamma}}&{=}
		s(\sigma)^* \beta^{-1} \bigl( \beta\gamma(b) v(\sigma) w(\sigma) v'(\sigma)^* \bigr) s'(\sigma)
		=
		b u(\sigma).
	\end{align*}
	Furthermore, it is straightforwardly checked that
	\begin{gather}\label{eq:proof_cohomologous}
		\beta^{-1} ( w(\sigma)^* v(\sigma^*) ) s(\sigma) u(\sigma) = \beta^{-1} ( v'(\sigma)^* ) s'(\sigma)
	\end{gather}
	for all $\sigma \in \hat G$.
	For each $(\sigma,\pi) \in \hat G \times \hat G$ we may thus compute
	\begin{alignat*}{2}
		u'(\sigma, \pi)
		&=
		&&s'(\sigma + \pi)^* \beta^{-1} \bigl( v'(\sigma + \pi) \omega'(\pi, \sigma)^* \gamma_\pi ( v'(\sigma)^* ) v'(\pi)^* \bigr) s'(\pi) s'(\sigma)
		\\
		\overset{\eqref{eq:comm_gamma}}&{=}
		&&s'(\sigma + \pi)^* \beta^{-1} \bigl( v'(\sigma + \pi) \omega'(\pi, \sigma)^* v'(\pi)^* \bigr) \gamma_\pi \beta^{-1} ( v'(\sigma)^* ) s'(\pi) s'(\sigma)
		\\
		&=
		&&s'(\sigma + \pi)^* \beta^{-1} \bigl( v'(\sigma + \pi) \omega'(\pi, \sigma)^* v'(\pi)^* \bigr) s'(\pi) \beta^{-1} ( v'(\sigma)^* ) s'(\sigma)
		\\
		\overset{\eqref{eq:proof_cohomologous}}&{=}
		&&u(\sigma+\pi)^* s(\sigma + \pi)^* \beta^{-1} \bigl( v(\sigma + \pi) w(\sigma+\pi)^*
		\omega'(\pi, \sigma)^*
		w(\pi)^* v(\pi)^* \bigr) s(\pi) u(\pi)
		\\
		&{} &&\quad \times
		\beta^{-1} ( w(\sigma)^* v(\sigma)^* ) s(\sigma) u(\sigma).
	\end{alignat*}
	Since $u(\sigma)$ and $u(\pi)$ are central in $\aB$, we may collect the terms $u(\sigma+\pi)$, $u(\sigma)$, and $u(\pi)$.
	Traversing the arguments for the first equalities further yields
	\begin{align*}
		u'(\sigma, \pi)
		&=
		\Delta_\sigma ( u(\pi) ) u(\sigma) u(\sigma+\pi)^*
		\cdot s(\sigma + \pi)^* \beta^{-1}
		\\
		&\quad \times\bigl( v(\sigma + \pi)
		w(\sigma+\pi) \omega'(\pi, \sigma)^* \gamma_\pi ( w(\sigma) )^* w(\pi)^*
		\cdot
		\gamma_\pi ( v(\sigma)^* ) v(\pi)^* \bigr) s(\pi) s(\sigma)
		\\
		&=
		\Delta_\sigma ( u(\pi) ) u(\sigma) u(\sigma+\pi)^* \cdot u(\sigma, \pi).
	\end{align*}
	This shows that $u$ and $u'$ are cohomologous as asserted.
\end{proof}

\begin{Lemma}\label{lem:cocycel_trivial}
	Suppose that $u$ is a $2$-coboundary for some $1$-cochain $u\colon \hat G \to \mathcal{U}Z(\aB)$, i.e.,
	\begin{gather}\label{eq:coboundary}
		u(\sigma,\pi) = \Delta_\sigma ( u(\pi) ) u(\sigma) u(\sigma + \pi)^*
	\end{gather}
	for all $\sigma, \pi \in \hat G$.
	Then
\[ \big(\hH, \gamma^\beta, \omega^\beta\big) = v'(\hH, \gamma, \omega)(v')^*\] for the partial isometries $v'(\sigma) := \beta\gamma_\sigma ( u(\sigma) ) v(\sigma)$, $\sigma \in \hat G$.
\end{Lemma}

\begin{proof}
	Using that each $u(\sigma)$, $\sigma \in \hat G$, is unitary in $\aB$, we see at once that $v'(\sigma)$, $\sigma \in \hat G$, are partial isometries in $\aB \tensor \End(\hH_\sigma)$, $\sigma \in \hat G$.
	Since $u(0) = \one_\aB$, it is also clear that $v'(0) = \one_\aB$.
	We~proceed to prove that $\big(\hH, \gamma^\beta, \omega^\beta\big) = v'(\hH, \gamma, \omega)(v')^*$.
	For this, we first observe that equation~\eqref{eq:unitary_central_element} can be rewritten as
	\begin{gather*}
		v(\sigma+\pi) \omega(\pi, \sigma)^* \gamma_\pi ( v(\sigma)^* ) v(\pi)^*
		=
		\beta ( s(\sigma+\pi) u(\sigma, \pi) s(\sigma)^* s(\pi)^* )
		\\
		\qquad{} \overset{\eqref{eq:coboundary}} {=}
		\beta \bigl( s(\sigma+\pi) u(\sigma+\pi)^* u(\sigma) \Delta_\sigma ( u(\pi) ) s(\sigma)^* s(\pi)^* \bigr)
		=
		(\omega')^\beta(\pi, \sigma)^*
	\end{gather*}
	for all $\sigma, \pi \in \hat G$.
	Combining this with the fact that the partial isometries $v(\sigma)$, $\sigma \in \hat G$, satisfy equation~\eqref{eq:comm_gamma} gives $v(\hH, \gamma, \omega)v^* = \bigl( \hH, \gamma^\beta, (\omega')^\beta \bigr)$.
	Next, for each $\sigma \in \hat G$ we put $s'(\sigma) := s(\sigma) u(\sigma)^* = \gamma_\sigma \bigl( u(\sigma)^* \bigr) s(\sigma)$ and note that $s'(\sigma)$ is an isometry in $\aA \tensor \End(V_\sigma,\hH_\sigma)$ satisfying $\alpha_g\big(s'(\sigma)\big)=s'(\sigma) \cdot \sigma_g$ for all $g \in G$.
	Evidently, $s'(0) = \one_\aA$.
	We write $(\hH, \gamma', \omega')$ for the associated factor system.
	As each $u(\sigma)$, $\sigma \in \hat G$, is central in $\aB$ we see that $\gamma' = \gamma$.
	Moreover, by construction, $\gamma(u)(\hH, \gamma, \omega')\gamma(u^*) = (\hH, \gamma, \omega)$ (cf.~\cite[Lemma~4.3]{SchWa17c}).
	It follows that $\beta \gamma(u) \bigl( \hH, \gamma^\beta, (\omega')^\beta \bigr) \beta \gamma(u^*) = \big(\hH, \gamma^\beta, \omega^\beta\big)$.
	We may thus conclude that
	\begin{gather*}
		v'(\hH, \gamma, \omega) (v')^*
		= \beta \gamma(u) v (\hH, \gamma, \omega) v^* \beta \gamma(u^*)
		= \big(\hH, \gamma^\beta, \omega^\beta\big).
		\tag*{\qed}
	\end{gather*}
	\renewcommand{\qed}{}
\end{proof}

In summary, we have thus proved:

\begin{Theorem}\label{thm:lift_cohomology}
	Let $G$ be a compact Abelian group, let $(\aA, G, \alpha)$ be a free C\Star-dynamical system with fixed point algebra $\aB$, and let $\beta$ be a \Star-automorphism of $\aB$.
	Suppose that $\gamma$ and $\gamma^\beta$ and are conjugated and let $u$ denote the corresponding $2$-cocycle.
	Then $\beta$ lifts to $\aA$ if and only if the cohomology class of $u$ in $ H^2\big(\hat G,\mathcal{U}Z(\aB)\big)_{\Delta}$ vanishes.
\end{Theorem}
\begin{proof}The assertion that the cohomology class of $u$ in $H^2\big(\hat G,\mathcal{U}Z(\aB)\big)_{\Delta}$ vanishes if $\beta$ lifts to $\aA$ follows from the introductory considerations in this section.
	The converse follows from Lemma~\ref{lem:cocycel_trivial} when combined with Theorem~\ref{thm:lift_conjugated_factor_sys}.
\end{proof}

\begin{Corollary}\label{cor:torus_action}
	Let $(\aA, \mathbb{T}, \alpha)$ be a free C\Star-dynamical system with fixed point algebra~$\aB$.
	Identify~$\hat{\mathbb T}$ with $\Z$ and write $1 \in \Z$ for its positive generator.
	Furthermore, let $\beta$ be a \Star-automorphism of~$\aB$ such that~$\gamma_1$ and~$\gamma_1^\beta$ are conjugated for some factor system $(\hH, \gamma, \omega)$ of $(\aA, \mathbb T, \alpha)$.
	Then $\beta$ lifts to~$\aA$.
\end{Corollary}
\begin{proof}
	The hypothesis does not depend on the choice of the factor system.
	For this reason, as a first step, we may choose the following system.
	Let $s \in \aA \tensor \End(\C,\hH_1)$ be an isometry satisfying $\alpha_z ( s ) = s (\one_{\aA} \tensor z)$ for all $z \in \mathbb T$.
	For each $n > 1$ this element yields an isometry $s(n) \in \aA \tensor \End(\C,\hH_1 \tensor \cdots \tensor \hH_1)$
	satisfying $\alpha_z ( s(n) ) = s(n) (\one_{\aA} \tensor z^n)$ for all $z \in \mathbb T$ by putting $s(n) := s_{1,2} s_{1,3} \cdots s_{1,n+1}$.
	Here and subsequently, the subindices refer to the leg numbering in the respective tensor product.
	It is easily checked that the coactions associated with the family~$s(n)$, $n\geq 0$, satisfy $\gamma_{m+n} = (\gamma_m \tensor \id) \circ \gamma_n$ for all $m,n \geq 0$.
	Moreover, for each $n > 1$ we put
	\begin{gather*}
		v(n) :=
		v_{1,2} s_{1,2} ( \cdots ( v_{1,n} s_{1,n}(v_{1,n+1}) s^*_{1,n} ) \cdots ) s^*_{1,2}
	\end{gather*}
	and note that $v(n) \in \aB \tensor \End(\hH_1 \tensor \cdots \tensor \hH_1)$.
	A straightforward induction now proves that $v(n)$ satisfies
	equation~\eqref{eq:comm_gamma} for all $n > 1$.
	Since similar considerations apply to the partial adjoint $s(-1) := (s^*)^T \in \aA \tensor \End(\C,\hH_1)$,
	we ultimately obtain partial isometries $v(n)$, $n \in \Z$, normalized to $v(0) := \one_\aB$, satisfying equation~\eqref{eq:comm_gamma}.
	Hence the assertion follows from combining Theorem~\ref{thm:lift_cohomology} with the fact that $H^2 (\Z,\mathcal{U}Z(\aB) )_{\Delta}$ vanishes (cf.~\cite[Chapter~VI.6]{Mac95}).
\end{proof}

In the remainder of this section we additionally assume that $\aA$ is commutative.
In this case, the Fr\"ohlich map $\Delta$ is trivial.
Furthermore, for each $\sigma \in \hat G$ we have $\gamma_\sigma(b) = b \gamma_\sigma(\one_\aB)$ for all $b \in \aB$, and for each $(\sigma,\pi) \in \hat G \times \hat G$ it may be derived that $\omega(\pi, \sigma) = \text{flip} ( \omega(\sigma,\pi) )$, where $\text{flip}$ stands for the tensor flip of $\End(\hH_\sigma) \tensor \End(\hH_\pi)$.
In consequence, the 2-cocycle $u\colon \hat G \times \hat G \to \mathcal{U}Z(\aB)$ from Lemma~\ref{lem:unitary_central_cocycle} above is untwisted and symmetric, and hence a 2-coboundary by~\cite[Lemma~3.6]{BCER84}.

\begin{Corollary}\label{cor:commAbelian}
	Suppose that $\aA$ is commutative. Let $G$ be a compact Abelian group, let $(\aA, G, \alpha)$ be a free C\Star-dynamical system with fixed point algebra $\aB$, and let $\beta$ be a \Star-automorphism of~$\aB$.
	Then~$\beta$ lifts to~$\aA$ if and only if $\beta\gamma_\sigma(\one_\aB)$ is Murray--von~Neumann equivalent to $\gamma_\sigma(\one_\aB)$ in \mbox{$\aB \tensor \End(\hH_\sigma)$} for all $\sigma \in \hat G$.
\end{Corollary}
\begin{Remark}\quad
	\begin{enumerate}\itemsep=0pt
	\item
		Turning to $K_0(\aB)$, we see that the ``if'' condition in Corollary~\ref{cor:commAbelian} is equivalent to the condition that $K_0(\beta)$ fixes the elements $[\gamma_\sigma(\one_\aB)] \in K_0(\aB)$, $\sigma \in \hat G$, which form a subgroup of $K_0(\aB)$ by Remark~\ref{rem:K0_factor_sys}.
	\item
	Let $\aA = \Cont(P)$ and $\aB = \Cont(X)$ for some compact spaces $P$ and $X$, respectively, and consider~$P$ as a topological principal $G$-bundle over~$X$.
Let $h\colon X \to X$ be the homeomorphism such that $\beta(f) = f \circ h$ for all $f \in \Cont(X)$.
Then the ``if'' condition in Corollary~\ref{cor:commAbelian} states that, for each $\sigma \in \hat G$, the vector bundles determined by $\gamma_\sigma(\one_\aB)$ and $\gamma_\sigma(\one_\aB) \circ h$, respectively, are equivalent.
	\end{enumerate}
\end{Remark}
\begin{proof}
	The ``only if'' direction is clear from Theorem~\ref{thm:lift_conjugated_factor_sys}.
	For the converse direction let, for each $\sigma \in \hat G$, $v(\sigma) \in \aB \tensor \End(\hH_\sigma)$ be a partial isometry such that $v(\sigma)^* v(\sigma) = \gamma_\sigma(\one_\aB)$ and $v(\sigma) v(\sigma)^* = \beta\gamma_\sigma(\one_\aB)$.
	By Theorem~\ref{thm:lift_cohomology} and the discussion prior to the corollary, it remains to prove that equation~\eqref{eq:comm_gamma} holds for all $\sigma \in \hat G$.
	Thus, let $\sigma \in \hat G$.
	Then for each $b \in \aB$ we have
	\begin{align*}
		\gamma^\beta_\sigma(b)
		&=
		\beta\big( \beta^{-1}(b) \gamma_\sigma(\one_\aB) \big)
		= b \beta ( \gamma_\sigma(\one_\aB) )
		= b v(\sigma) v(\sigma)^*
		= v(\sigma) b v(\sigma)^*
		\\
		&=
		v(\sigma) \cdot b\gamma_\sigma(\one_\aB) \cdot v(\sigma)^*
		= v(\sigma) \cdot \gamma_\sigma(b) \cdot v(\sigma)^*
		= ( \Ad[v(\sigma)] \circ \gamma_\sigma ) (b).
	\end{align*}
	By conjugating with $v(\sigma)^*$, we also get $\gamma_\sigma = \Ad[v(\sigma)^*] \circ \gamma_\sigma^\beta$.
\end{proof}

\section{Lifting 1-parameter groups}\label{sec:1-para}

Let $(\beta_t)_{t \in \R}$ be a smooth 1-parameter group of \Star-automorphisms $\beta_t \in \Aut(\aB)$ and let $\delta := D\beta_t$ denote the corresponding \Star-derivation on its smooth domain $\aB^\infty$ (cf.\ Section~\ref{sec:pre}).
In this section we investigate whether there is a smooth 1-parameter group $\big(\hat\beta_t\big)_{t \in \R}$ of \Star-automorphisms $\hat\beta_t \in \Aut(\aA)$ such that, for each $t \in \R$, $\hat\beta_t$ is a lift of $\beta_t$.
In the affirmative case we say that $(\beta_t)_{t \in\R}$ \emph{lifts smoothly} to $\aA$ and that $(\hat\beta_t)_{t \in \R}$ is a \emph{smooth lift} of $(\beta_t)_{t \in \R}$.

\begin{Example}\quad
	\begin{enumerate}\itemsep=0pt
	\item
		We recall from the classical theory of smooth principal bundles that every smooth 1-pa\-ra\-me\-ter group on the base manifold lifts (smoothly) to the total space of the principal bundle.
		For a compact Abelian group $G$ this follows from Corollary~\ref{cor:commAbelian}, because, for each $\sigma \in \hat G$, the projections $\gamma_\sigma(\one_\aB)$ and $\beta\gamma_\sigma(\one_\aB)$ are obviously homotopic, and hence Murray--von~Neumann equivalent.
	\item
		The example in \cite[Section~4]{BCER84} shows that not all 1-parameter groups lift.
	\end{enumerate}
\end{Example}

Let us fix, for each $\sigma \in \Irrep(G)$, a finite-dimensional Hilbert space $\hH_\sigma$ and an isometry $s(\sigma)\in \aA \tensor \End(V_\sigma,\hH_\sigma)$ satisfying $\alpha_g\big(s(\sigma)\big)=s(\sigma) \cdot \sigma_g$ for all $g \in G$; for $1 \in \Irrep(G)$ we take $\hH_1 := \C$ and $s(1) := \one_\aA$.
Throughout the following, we make the standing assumption that the associated factor system $(\hH, \gamma, \omega)$ is smooth in the sense that
\begin{gather*}
	\gamma_\sigma(\aB^\infty)
	\subseteq \aB^\infty \tensor \End(\hH_\sigma)
\qquad
	\text{and}
\qquad
	\omega(\sigma, \pi)
	\in \aB^\infty \tensor \End(\hH_{\sigma \tensor \pi},\hH_\sigma \tensor \hH_\pi)
\end{gather*}
for all $\sigma, \pi \in \Irrep(G)$.

\begin{Lemma}\label{lem:1_param_to_H}
	Let $\big(\hat\beta_t\big)_{t \in \R}$ be a smooth lift of $(\beta_t)_{t \in \R}$ and let $\hat \delta := D\hat\beta_t$ denote the corresponding \Star-derivation on its smooth domain $\aA^\infty$.
	Suppose that, for each $\sigma \in \Irrep(G)$, the isometry $s(\sigma)$ is smooth for $\big(\hat\beta_t\big)_{t \in \R}$, i.e., $s(\sigma) \in \aA^\infty \in \End(V_\sigma, \hH_\sigma)$.
	Then the element
	\begin{gather*}
		H(\sigma) := \hat \delta ( s(\sigma) ) s(\sigma)^* .
	\end{gather*}
	lies in $\aB^\infty \tensor \End(\hH_\sigma)$ and for all $\sigma, \pi \in \Irrep(G)$ we have
	\begin{gather}
		\label{eq:deriv_cond_1}
		\delta \gamma_\sigma(b)
		= \gamma_\sigma \delta(b) + H(\sigma) \gamma_\sigma(b) + \gamma_\sigma(b) H(\sigma)^*
		\qquad \forall b \in \aB^\infty,
		\\
		\label{eq:deriv_cond_2}		
		\delta ( \omega(\sigma,\pi) )
		= H(\sigma) \omega(\sigma,\pi) + \gamma_\sigma (H(\pi) ) \omega(\sigma,\pi) + \omega(\sigma,\pi) H(\sigma \tensor \pi)^*.
	\end{gather}
\end{Lemma}
\begin{proof}
	Let $\sigma \in \Irrep(G)$.
	We first note that $\aA^\infty$ is a \Star-subalgebra of $\aA$ satisfying $\hat \delta(\aA^\infty) \subseteq \aA^\infty$ and $(\aA^\infty)^G = \aB^\infty$, the latter being due to the fact that $\big(\hat \beta_t\big)_{t \in \R}$ is a lift of $(\beta_t)_{t \in \R}$.
	Therefore the assumption on $s(\sigma)$ implies that $H(\sigma) = \hat \delta ( s(\sigma) ) s(\sigma)^* \in \aB^\infty \tensor \End(\hH_\sigma)$.
	Furthermore, for each $t \in \R$ and $b \in \aB$ it follows from the lifting property that
	\begin{gather*}
		\beta_t ( \gamma_\sigma(b) )
		= \beta_t ( s(\sigma) b s(\sigma)^* )
		= \hat\beta_t ( s(\sigma) ) \beta_t(b) \hat\beta_t ( s(\sigma) )^*.
	\end{gather*}
	Taking the derivative at $t=0$, for each $b \in \aB^\infty$ we get equation~\eqref{eq:deriv_cond_1}:
	\begin{align*}
		\delta \gamma_\sigma(b)
		&=
		\hat \delta ( s(\sigma) ) b s(\sigma)^* + s(\sigma) \delta(b) s(\sigma)^* + s(\sigma) b \hat \delta ( s(\sigma) )^*
		\\
		&= H(\sigma) \gamma_\sigma(b) + \gamma_\sigma \delta(b) + \gamma_\sigma(b) H(\sigma)^*.
	\end{align*}
	Now, let $\sigma, \pi \in \Irrep(G)$.
	Just as above, for each $t \in \R$ we find that
	\begin{gather*}
		\beta_t ( \omega(\sigma, \pi) )
		= \hat \beta_t ( s(\sigma) ) \hat\beta_t ( s(\pi) ) \hat\beta_t ( s(\sigma \tensor \pi) )^*.
	\end{gather*}
	Taking the derivative at $t=0$ yields equation~\eqref{eq:deriv_cond_2}:
	\begin{align*}
		\delta ( \omega(\sigma, \pi) )
		&=
		\hat \delta ( s(\sigma) ) s(\pi) s(\sigma \tensor \pi)^*
		+ s(\sigma) \hat \delta ( s(\pi) ) s(\sigma \tensor \pi)^*
		+ s(\sigma) s(\pi) \hat \delta ( s(\sigma \tensor \pi) )^*
		\\
		&=
		H(\sigma) \omega(\sigma, \pi)
		+ \gamma_\sigma ( H(\pi) ) \omega(\sigma, \pi)
		+ \omega(\sigma, \pi) H(\sigma \tensor \pi)^*.
		\tag*{\qed}	
	\end{align*}
	\renewcommand{\qed}{}
\end{proof}

\begin{Lemma}\label{lem:H_to_1_param}
	Let $H(\sigma) \in \aB^\infty \tensor \End(\hH_\sigma)$, $\sigma \in \Irrep(G)$, be a family, normalized to $H(1) = 0$, satisfying equations~\eqref{eq:deriv_cond_1} and~\eqref{eq:deriv_cond_2}.
	Then there is a smooth lift $\big(\hat \beta_t\big)_{t \in \R}$ of $(\beta_t)_{t \in \R}$ with
	\begin{gather*}
		D\hat \beta_t \bigl( \Tr ( y s(\sigma) ) \bigr) = \Tr \bigl( D\beta(y) s(\sigma) + x H(\sigma) s(\sigma) \bigr)
	\end{gather*}
	for all $y \in \aB^\infty \tensor \End(\hH_\sigma, V_\sigma)$ and $\sigma \in \Irrep(G)$.
\end{Lemma}
\begin{proof}
	Without loss of generality we may assume that $H(\sigma) = H(\sigma) \gamma_\sigma(\one_\aB)$ for all $\sigma \in \Irrep(G)$, because replacing $H(\sigma)$ by $H(\sigma) \gamma_\sigma(\one_\aB)$ does neither change equation~\eqref{eq:deriv_cond_1} nor equation~\eqref{eq:deriv_cond_2}.
	The task is now to prove that $\big(\hH, \gamma^{\beta_t}, \omega^{\beta_t}\big) \sim (\hH, \gamma, \omega)$ for all $t \in \R$.
	For this purpose, let $\sigma \in \Irrep(G)$.
	To handle the first conjugacy condition, we examine the differential equation $\dot v_t(\sigma) = \beta_t ( H(\sigma) ) v_t(\sigma)$ in $\aB \tensor \End(\hH_\sigma) \gamma_\sigma(\one_\aB)$ with initial condition $v_0(\sigma) = \gamma_\sigma(\one_\aB)$.
	Indeed, by~\cite[Section~3]{daleckii02}, it admits a unique solution $v_t(\sigma) \in \aB \tensor \End(\hH_\sigma) \gamma_\sigma(\one_\aB)$ satisfying the $\beta$-cocycle equations
	\begin{gather}\label{eq:proof_1_cocycle}
		v_{s+t}(\sigma)
		= \beta_t ( v_s(\sigma) ) v_t(\sigma)
		\qquad
		\text{and}
		\qquad
		v_{-t}(\sigma)
		= \beta_t^{-1} ( v_t(\sigma)^* )
	\end{gather}
for all $s, t \in \R$.
	In particular, each $v_t(\sigma)$, $t \in \R$, is a partial isometry with initial and final projection given by $v_t(\sigma)^* v_t(\sigma) = \gamma_\sigma(\one_\aB)$ and $v_t(\sigma) v_t(\sigma)^* = \beta_t ( \gamma_\sigma(\one_\aB) )$, respectively.
	Furthermore, for each $b \in \aB^\infty$ we consider the function $t \mapsto b_t := \gamma^{\beta_t}_\sigma(b) = \beta_t\gamma_\sigma\beta_t^{-1}(b)$ which is clearly smooth and satisfies the differential equation
	\begin{align*}
		\dot b_t
		&
		= \beta_t\delta \gamma_\sigma\beta_t^{-1}(b) - \beta_t\gamma_\sigma\delta \beta_t^{-1}(b)
		= \beta_t\bigl( \delta \gamma_\sigma \beta_t^{-1}(b) - \gamma_\sigma \delta \beta_t^{-1}(b) \bigr)
		\\
		\overset{\eqref{eq:deriv_cond_1}}&=
		\beta_t \bigl( H(\sigma) \gamma_\sigma \beta_t^{-1}(b) + \gamma_\sigma\beta_t^{-1}(b) H(\sigma)^* \bigr)
		=
		\beta_t ( H(\sigma) ) b_t + b_t \beta_t ( H(\sigma) )
	\end{align*}
	for all $t \in \R$ with initial condition $b_0 = \gamma_\sigma(b)$.
	As the same equation and initial condition are satisfied by the function $t \mapsto v_t(\sigma) \gamma_\sigma(b) v_t(\sigma)^*$ we must have
	\begin{gather}		\label{eq:proof_conjugation_gamma}
		\gamma_\sigma^{\beta_t}(b) = v_t(\sigma) \gamma_\sigma(b) v_t(\sigma)^*
	\end{gather}
	for all $t \in \R$ and $b \in \aB$.
	By conjugating with $v_t(\sigma)^*$, we also obtain $\gamma_\sigma = \Ad[v_t(\sigma)^*] \circ \gamma_\sigma^{\beta_t}$.
	Next, let $\sigma, \pi \in \Irrep(G)$.
	To deal with the second conjugacy condition, we look at the function $t \mapsto \omega_t := v_t(\sigma) \gamma_\sigma ( v(\pi) ) \omega(\sigma, \pi) v_t(\sigma \tensor \pi)^*$.
	Evidently, this function is smooth.
	Moreover, it satisfies the differential equation
	\begin{align*}
		\dot \omega_t
		&=
		\dot v_t(\sigma) \gamma_\sigma ( v_t(\pi) ) \omega v_t(\sigma \tensor \pi)^*
		+ v_t(\sigma) \gamma_\sigma ( \dot v_t(\pi) ) \omega v_t(\sigma \tensor \pi)^*
		+ v_t(\sigma) \gamma_\sigma ( v_t(\pi) ) \omega \dot v_t(\sigma \tensor \pi)^*
		\\
		&=
		\beta_t ( H(\sigma) ) \omega_t
		+ v_t(\sigma) \gamma_\sigma \beta_t ( H(\pi) ) v_t(\sigma)^* \omega_t
		+ \omega_t H(\sigma \tensor \pi)^*
		\\
		\overset{\eqref{eq:proof_conjugation_gamma}}&=
		\beta_t ( H(\sigma) ) \omega_t
		+ \beta_t\gamma_\sigma( ( H(\pi) ) \omega_t
		+ \omega_t H(\sigma \tensor \pi)^*
	\end{align*}
	for all $t \in \R$ with initial condition $\omega_0 = \omega(\sigma, \pi)$.
	The same equation and initial condition are satisfied by the function $t \mapsto \omega^{\beta_t}(\sigma, \pi) = \hat\beta_t ( s(\sigma) s(\pi) s(\sigma \tensor \pi)^* )$. It follows that
	\begin{gather*}
		\omega^{\beta_t}(\sigma,\pi)
		= v_t(\sigma) \gamma_\sigma ( v(\pi) ) \omega(\sigma, \pi) v_t(\sigma \tensor \pi)^*
	\end{gather*}
	for all $t \in \R$.
	Summarizing, we have shown that $(\hH, \gamma^{\beta_t}, \omega^{\beta_t}) = v_t(\hH, \gamma, \omega)v_t^*$ for all $t \in \R$.
	Hence, for each $t \in \R$, Theorem~\ref{thm:lift_conjugated_factor_sys} provides us with a lift $\hat\beta_t$ of $\beta_t$ satisfying
	\begin{gather}\label{eq:proof_lift}
		\hat\beta_t \bigl( \Tr ( y s(\sigma) ) \bigr) = \Tr \bigl( \beta_t(y) v_t(\sigma) s(\sigma) \bigr)
	\end{gather}
	for all $y \in \aB \tensor \End(\hH_\sigma, V_\sigma)$ and $\sigma \in \Irrep(G)$ (cf.\ equation~\eqref{eq:beta}).
	In addition, equations~\eqref{eq:proof_1_cocycle} imply that $(\hat\beta_t)_{t \in \R}$ constitutes a 1-parameter group as required.
	Finally, it is immediate from equation~\eqref{eq:proof_lift} that, for each $y \in \aB^\infty \tensor \End(\hH_\sigma, V_\sigma)$, the element $\Tr \bigl( ys(\sigma) \bigr)$ is smooth and that
	\begin{align*}
		D\hat\beta\bigl( \Tr( ys(\sigma) \bigr)
		&
		= \Tr \bigl( D\beta(y) v_0(\sigma) s(\sigma) \bigr) + \Tr \bigl( y \dot v_0(\sigma) s(\sigma) \bigr)
		\\
		\overset{\eqref{eq:ranges_sys}}&=
		\Tr \bigl( D\beta(y) s(\sigma) \bigr) + \Tr \bigl( y H(\sigma) s(\sigma) \bigr)
	\end{align*}
	for all $\sigma \in \Irrep(G)$ and $y \in \aB^\infty \tensor \End(\hH_\sigma, V_\sigma)$.
\end{proof}

Combining Lemmas~\ref{lem:1_param_to_H} and~\ref{lem:H_to_1_param}, we have established:

\begin{Theorem}\label{thm:lift_1_param}
	Let $(\aA, G, \alpha)$ be a free C\Star-dynamical system with fixed point algebra $\aB$.
	Furthermore, let $(\beta_t)_{t \in \R}$ be a smooth $1$-parameter group of \Star-automorphisms $\beta_t \in \Aut(\aB)$.
	Then the following statements are equivalent:
	\begin{enumerate}\itemsep=0pt
	\item[$(a)$]
		$(\beta_t)_{t \in \R}$ lifts smoothly to $\aA$.
	\item[$(b)$]
		There is a family $H(\sigma) \in \aB^\infty \tensor \End(\hH_\sigma, V_\sigma)$, $\sigma \in \Irrep(G)$, normalized to $H(1)= 0$, satisfying equations~\eqref{eq:deriv_cond_1} and~\eqref{eq:deriv_cond_2} for all $\sigma, \pi \in \Irrep(G)$ and $b \in \aB^\infty$.
	\end{enumerate}
\end{Theorem}

\section{An Atiyah sequence for noncommutative principal bundles}\label{sec:Atiyah}

In this section we generalize the classical Atiyah sequence in equation~\eqref{eq:Atiyah} to the setting of free C\Star-dynamical systems.
In addition, we explain how this can be used to produce characteristic classes.

To this end, we consider a free C\Star-dynamical system $(\aA,G,\alpha)$ with fixed point algebra~$\aB$ and we fix a dense unital \Star-subalgebra $\aB_0 \subseteq \aB$.
Again, for each $\sigma \in \Irrep(G)$ we choose a finite-dimensional Hilbert space $\hH_\sigma$ and an isometry $s(\sigma)$ in $\aA \tensor \End(V_\sigma,\hH_\sigma)$ satisfying $\alpha_g(s(\sigma))=s(\sigma) (\one_\aA \tensor \sigma_g)$ for all $g \in G$; for $1 \in \Irrep(G)$, we take $\hH_1 := \C$ and $s(1) := \one_\aA$.
We denote by $(\hH, \gamma, \omega)$ the associated factor system and assume that
\begin{gather*}
	\gamma_\sigma(\aB_0) \subseteq \aB_0 \tensor \End(\hH_\sigma)
	\qquad \text{and}
\qquad
\omega(\sigma, \pi) \in \aB_0 \tensor \End(\hH_{\sigma \tensor \pi},\hH_\sigma \tensor \hH_\pi)
\end{gather*}
for all $\sigma, \pi \in \Irrep(G)$.
Similar arguments as in~\cite[Section~5.1]{SchWa20} establish that
\begin{gather*}
	\aA_0 := \Span \bigl\{ \Tr ( y s(\sigma) ) \colon \sigma \in \Irrep(G), \, y \in \aB_0 \tensor \End(\hH_\sigma,V_\sigma) \bigr\}
\end{gather*}
is a dense, $\alpha$-invariant, and unital \Star-subalgebra of $\aA$ such that $\aA_0^G = \aB_0$ (cf.\ Section~\ref{sec:facsys}).
The algebra~$\aA_0$ depends on the choice of the isometries $s(\sigma)$, $\sigma \in \Irrep(G)$.
However, any other choice of isometries $s'(\sigma)$, $\sigma \in \Irrep(G)$, such that $s'(\sigma) \in \aA_0 \tensor \End(V_\sigma, \hH_\sigma')$ yields the same subalgebra.
In~fact, $\aA_0$ is the smallest \Star-subalgebra of $\aA$ such that $\aA_0 \supseteq \aB_0$ and $s(\sigma) \in \aA_0 \tensor \End(V_\sigma, \hH_\sigma)$ for all $\sigma \in \Irrep(G)$.

\subsection{The associated Atiyah sequence}

For a \Star-derivation $\delta$ on $\aB_0$ we say that $\delta$ \emph{lifts to $\aA_0$} if there is a \Star-derivation $\hat \delta$ on $\aA_0$ that extends $\delta$ and commutes with all $\alpha_g$, $g \in G$.
In this case we call $\hat \delta$ a \emph{lift of $\delta$}.
As an immediate consequence of Theorem~\ref{thm:lift_1_param} we obtain:
\begin{Corollary}\label{cor:liftder}
	Let $\delta \in \der(\aB_0)$.
	Then the following statements are equivalent:
\begin{enumerate}\itemsep=0pt
\item[$(a)$]
			$\delta$ lifts to $\aA_0$.
\item[$(b)$] 
There is a family $H(\sigma) \in \aB_0 \tensor \End(\hH_\sigma, V_\sigma)$, $\sigma \in \Irrep(G)$, normalized to $H(1)= 0$, satisfying equations~\eqref{eq:deriv_cond_1} and~\eqref{eq:deriv_cond_2} for all $\sigma, \pi \in \Irrep(G)$ and $b \in \aB_0$.
	\end{enumerate}
	Moreover, under the assumptions of $(b)$ a lift of $\delta$ is given by linearly extending
	\begin{gather*}
		\hat \delta\bigl( \Tr ( xs(\sigma) ) \bigr) := \Tr\bigl( \delta(x)s(\sigma) + x H(\sigma) s(\sigma) \bigr)
	\end{gather*}
	for $\sigma \in \Irrep(G)$ and $x \in \aB_0 \tensor \End(\hH_\sigma,V_\sigma)$.
\end{Corollary}

\begin{Remark}\label{rem:cleftAbelian}
Suppose that $G$ is compact Abelian and that $(\aA,G,\alpha)$ is cleft.
This is essentially the setting studied by Batty, Carey, Evans, and Robinson in~\cite{BCER84}, but without the factor system terminology.
Combining Corollary~\ref{cor:liftder} with the results from Section~\ref{sec:cpt_abelian}, we obtain a generalization of the main results in~\cite{BCER84} to the setting of free actions of compact groups.
\end{Remark}

To formulate a generalized Atiyah sequence, we consider the Lie algebra
\begin{gather*}
	\der_G(\aA_0) := \{ \delta \in \der(\aA_0)\colon \alpha_g \circ \delta = \delta \circ \alpha_g \ \forall g \in G \}.
\end{gather*}
This Lie algebra admits the short exact sequence
\begin{gather*}\label{eq:AtiyahNC}
	0 \longrightarrow \gau(\aA_0) \longrightarrow \der_G(\aA_0)\longrightarrow \der(\aB_0)_{(\gamma,\omega)} \longrightarrow 0,
	\\
	\shortintertext{where}
	\gau(\aA_0) := \bigl\{\delta \in \der_G(\aA_0) : \delta_{\mid \aB_0} = 0 \bigr\}
\end{gather*}
is the Lie algebra of \emph{infinitesimal gauge transformations} of $\aA_0$ and $\der(\aB_0)_{(\gamma,\omega)}$ is the Lie subalgebra of all \Star-derivations of $\aB_0$ that lift to $\aA_0$.
By Corollary~\ref{cor:liftder}, we have
\begin{gather*}
	\der(\aB_0)_{(\gamma,\omega)} = \{ \delta \in \der(\aB_0)\colon \text{$\delta$ satisfies the condition in Corollary~\ref{cor:liftder}(b)} \}.
\end{gather*}
Furthermore, it follows from Corollary~\ref{cor:liftder} that $\gau(\aA_0)$ can be identified with the Lie algebra, let us say, $H(\hH, \gamma, \omega)$, consisting of all families of skew-symmetric elements $H(\sigma) \in \aB_0 \tensor \End(\hH_\sigma)$, $\sigma \in \Irrep(G)$, satisfying the equations
\begin{gather*}
		 H(\sigma) \gamma_\sigma(b) + \gamma_\sigma(b) H(\sigma)^* = 0
		\qquad 	\forall b \in \aB_0,\\
		 H(\sigma) \omega(\sigma,\pi) + \gamma_\sigma\big(H(\pi)\big) \omega(\sigma,\pi) + \omega(\sigma,\pi) H(\sigma \tensor \pi)^* = 0
	\end{gather*}
for all $\sigma, \pi \in \Irrep(G)$.

\begin{Remark}\label{rem:At_cleft}
Suppose that $(\aA,G,\alpha)$ is cleft with compact Abelian $G$ and that $\omega(\sigma,\pi)$ lies in $\C \cdot \one_\aB$ for all $\sigma,\pi \in \hat G = \Hom(G,\mathbb{T})$. Then $H(\hH, \gamma, \omega)$ may be realized as the Lie algebra of all crossed homomorphisms $\hat G \to \aB_0^{\text{skew}}$:
	\begin{gather*}
		Z^1\big( \hat G, \aB_0^{\text{skew}} \big) := \big\{ H\colon \hat G \to \aB_0^{\text{skew}}\colon H(\sigma+\pi) = H(\sigma) + \gamma_\sigma(H_\pi) \ \forall \sigma,\pi \in \hat G \big\}.
	\end{gather*}
\end{Remark}

\begin{Definition}\label{def:connection}
A linear section $\chi\colon \der(\aB_0)_{(\gamma,\omega)} \to \der_G(\aA_0)$ of the Atiyah sequence in equation~\eqref{eq:AtiyahNCgroup} is called a \emph{connection} of $\aA_0$.
\end{Definition}

\begin{Remark}
If $q\colon P \to M$ is a smooth principal bundle with structure group $G$, then the previous definition reproduces, up to a suitable completion, the classical setting of connection 1-forms when restricted to sections $\mathcal{V}(M) \to \mathcal{V}(P)^G$ that are $C^\infty(M)$-linear (see, e.g., \cite[Chapter~XII]{Kob69II}).
\end{Remark}

Given a connection $\chi\colon \der(\aB_0)_{(\gamma,\omega)} \to \der_G(\aA_0)$ of $\aA_0$, we may utilize Lecomte's Chern--Weil homomorphism to associate characteristic classes with $(\aA,G,\alpha)$.
More precisely, for each $k \in\mathbb{N}_0$ we get a natural map
\begin{gather*}
	C_k \colon \ \text{Sym}^k(\gau(\aA_0),\aB_0)^{\der_G(\aA_0)}\rightarrow H^{2k}\big( \der(\aB_0)_{(\gamma,\omega)},\aB_0\big), \qquad f \mapsto \frac{1}{k!}[f_\chi],
\end{gather*}
which, as a matter of fact, is independent of the choice of $\chi$.
Here, $\text{Sym}^k(\gau(\aA_0),\aB_0)$ stands for~the space of symmetric $k$-linear maps $\gau(\aA_0)^k \to \aB_0$ and $f_\chi$ is the $2k$-cocycle in \linebreak $C^{2k}\big(\der(\aB_0)_{(\gamma,\omega)},\aB_0\big)$ associated with $\chi$ (cf.~\cite{Lec85,Wa18}).

\subsection{An example: quantum 3-tori}\label{sec:Q3T}

Let $\theta$ be a real skew-symmetric $3\times 3$-matrix and, for $1 \le k, \ell \le 3$, put $\lambda_{k,\ell} := \exp(2\pi \imath \theta_{k,\ell})$ for short.
In the following we consider the quantum 3-torus $\mathbb A^3_\theta$, which is the universal C\Star-algebra with unitary generators $u_1$, $u_2$, $u_3$ satisfying the relation $u_k u_\ell = \lambda_{k,\ell} u_\ell u_k$ for all $1 \le k,\ell \le 3$.
The classical torus $\mathbb T^3$ acts naturally on $\mathbb A^3_\theta$ via the \Star-automorphisms given by $\tau(u_k) = z_k \cdot u_k$ for all $z = (z_1,z_1,z_3) \in \mathbb T^3$ and $1 \le k \le 3$.
This is the so-called \emph{gauge action}, whose generators are the \Star-derivations $\delta_k(u_\ell) = 2 \pi \imath \delta_{k, \ell} \cdot u_\ell$, $1 \le k,\ell \le 3$, where $\delta_{k,\ell}$ denotes the Kronecker delta.

Our study revolves around the restricted gauge action $\alpha\colon \mathbb T \to \Aut\big(\mathbb A^3_\theta\big)$ defined by
\begin{gather*}
	\alpha_z(u_1) := u_1,
\qquad
	\alpha_z(u_2) := u_2,
\qquad
	\alpha_z(u_3) :=z \cdot u_3
\end{gather*}
for all $z \in \mathbb T$.
Its fixed point algebra is the quantum 2-torus $\mathbb A^2_{\theta'}$ generated by the unitaries $u_1$ and $u_2$, where $\theta'$ denotes the real skew-symmetric $2\times 2$-matrix with upper right off-diagonal entry~$\theta_{12}$.
More generally, for each $k \in \Z$, the corresponding isotypic component is $\mathbb A^3_\theta(k)$ takes the form $u_3^k \mathbb A^2_{\theta'}$.
In particular, the C\Star-dynamical system $\big(\mathbb A^3_\theta, \mathbb T, \alpha\big)$ is cleft and therefore free.
The factor system associated with the unitaries $u(k) := u_3^k$, $k \in \Z$, is given by the following data:
For $k \in \Z $ we have $\gamma_k= \tau_z$ with $z = \big(\lambda_{31}^k, \lambda_{32}^k,1\big)$, and for $k,l \in \Z $ the cocycle $\omega(k,l)$ computes as~$\smash{\one_{\mathbb A^2_{\theta'}}}$.

Next, we look at the dense unital \Star-subalgebra $\aB_0$ of $\mathbb A^2_{\theta'}$ generated by all noncommutative polynomials in $u_1$ and $u_2$.
Then $\aA_0$ is given by the dense unital \Star-algebra of $\mathbb A^3_\theta$ generated by all noncommutative polynomials in $u_1$, $u_2$, and $u_3$, as is easily seen.
Furthermore, it follows from~\cite[Section~3.4]{DVMiKriMaMi01} that the \Star-derivations of $\aB_0$ split as a semidirect product of inner and outer \Star-derivations, i.e., $\der(\aB_0) \cong \text{Inn}(\der(\aB_0))\rtimes \text{Out}(\der(\aB_0))$.
If moreover $\theta_{12}$ is irrational, then $\text{Out}(\der(\aB_0))$ is linearly generated by ${\delta_1}_{\mid \aB_0}$ and ${\delta_2}_{\mid \aB_0}$.
Obviously, ${\delta_1}_{\mid \aB_0}$ and ${\delta_2}_{\mid \aB_0}$ may be lifted to ${\delta_1}_{\mid \aA_0}$ and~${\delta_2}_{\mid \aB_0}$, respectively.
It is also clear that each inner \Star-derivation of $\aB_0$ lifts to~$\aA_0$.
In~consequence, $\der(\aB_0)_{(\gamma,\omega)} = \der(\aB_0)$. Since we also have $Z^1\big( \Z, \aB_0^{\text{skew}} \big) \cong \aB_0^{\text{skew}}$ via the evaluation map $f \mapsto f(1)$, it follows that the associated Atiyah sequence reads as
\begin{gather*}
	0 \longrightarrow \aB_0^{\text{skew}} \longrightarrow \der_G(\aA_0) \longrightarrow \der(\aB_0) \longrightarrow 0
\end{gather*}
(cf.~Remark~\ref{rem:At_cleft}).
Finally, a moment's thought shows that this sequence is split, and so it does, unfortunately, only give trivial Chern--Weil--Lecomte classes.
However, if needed, one may associate secondary characteristic classes as described in~\cite{Wa18}.

\subsection{Associating connections and curvature}\label{sec:infobj}

Connection 1-forms are the fundamental tool in the theory of smooth principal bundles and give rise to the notion of connections on associated vector bundles.
Such a connection or, more precisely, its induced covariant derivative is an operator that can differentiate sections of each associated vector bundle along tangent directions in the base manifold.
	
In this section we discuss suitable generalizations of theses notions to the C\Star-algebraic setting of noncommutative principal bundles.
In particular, we provide explicit formulas for connections and curvature on associated noncommutative vector bundles.
To do this, we proceed as follows:

For each finite-dimensional representation $(\sigma,V_\sigma)$ of $G$ we consider the associated natural inner product $\aA_0 \tensor \End(V_\sigma)-\aB_0$-bimodule
\begin{gather*}
\Gamma_{\aA_0}(V_\sigma) := \Gamma_{\aA_0}(\sigma,V_\sigma) := \{ x \in \aA_0 \tensor V_\sigma\colon (\alpha_g \tensor \sigma_g)(x) = x \ \forall g\in G\}
	\end{gather*}
	with left $\aA_0 \tensor \End(V_\sigma)$-valued inner product~$\lprod{\aA_0 \tensor \End(V_\sigma)}{\cdot,\cdot}$ and right $\aB_0$-valued inner product $\rprod{\aB_0}{\cdot,\cdot}$
	defined on simple tensors by
	\begin{gather*}
		\lprod{\aA_0 \tensor \End(V_\sigma)}{a \tensor v, b \tensor w} := ab^* \tensor \ketbra{v}{w} \qquad \text{and} \qquad \rprod{\aB_0}{a \tensor v, b \tensor w} := \langle v,w \rangle a^*b 
	\end{gather*}
	respectively. Notably, the bimodule structure and the inner products are related by the compatibility condition $\lprod{\aA_0 \tensor \End(V_\sigma)}{x,y} \acts z = x \acts \rprod{\aB_0}{y,z}$ for all $x,y,z \in \Gamma_{\aA_0}(V_\sigma)$.
	
\begin{Remark}The space $\Gamma_{\aA_0}(V_\sigma)$ may be interpreted as associated noncommutative vector bundles (cf.~\cite{Wa12M}).
\end{Remark}

The following result provides a criterion for ensuring that the space $\Gamma_{\aA_0}(V_\sigma)$ admits a so-called standard right-module frame (see, e.g.,~\cite[Section~2]{Rieffel08}).

\begin{Lemma}\label{lem:partition}
	Let $(\sigma,V_\sigma)$ be a finite-dimensional representation of $G$.
	Then there are elements $s_1, \dots, s_d \in \Gamma_{\aA_0}(V_\sigma)$ such that
	\begin{gather*}
		\sum_{k=1}^d \lprod{\aA_0 \tensor \End(V_\sigma)}{s_k,s_k}= \one_{\aA \tensor \End(V_\sigma)}.
	\end{gather*}
	In particular, the reproducing formula $x = \sum_{k=1}^d s_k \acts \rprod{\aB_0}{s_k,x}$ holds for all $x \in \Gamma_{\aA_0}(V_\sigma)$.
	If $(\aA, G, \alpha)$ is cleft, we find such elements with $\rprod{\aB_0}{s_k,s_l} = \delta_{k,l} \cdot \one_\aB$ for all $1\leq k,l \leq d$.
\end{Lemma}
\begin{proof}
	Let $e_1,\dots,e_d$ be an orthonormal bases of $\hH_\sigma$ and let $s_k \in \aA_0 \tensor V_\sigma$, $1 \leq k \leq d$, be the columns of $s(\sigma)^* \in \aA_0 \tensor \End(\hH_\sigma,V_\sigma) $.
	Then $s_k \in \Gamma_{\aA_0}(V_\sigma)$ for all $1 \leq k \leq d$, which is due to the fact that $\alpha_g(s(\sigma)) = s(\sigma) (\one_\aA \tensor \sigma_g)$ for all $g \in G$.
	In addition, a moment's thought reveals that
	\begin{align*}
		\sum_{k=1}^d \lprod{\aA_0 \tensor \End(V_\sigma)}{s_k,s_k} = s(\sigma)^* s(\sigma) = \one_{\aA \tensor \End(V_\sigma)}.
	\end{align*}
	If $(\aA, G, \alpha)$ is cleft, then $\rprod{\aB_0}{s_k,s_l} = \delta_{k,l} \cdot \one_\aB$ for all $1\leq k,l \leq d$ follows from the fact that~$s(\sigma)$ is unitary. The verification of the reproducing formula poses no trouble and is left to the reader.
\end{proof}

In what follows, we consider a fixed finite-dimensional representation $(\sigma,V_\sigma)$ of $G$ and elements $s_1, \dots, s_d \in \Gamma_{\aA_0}(V_\sigma)$ as in Lemma~\ref{lem:partition}.

\begin{Lemma}\label{lem:conn}
	Let $(\sigma,V_\sigma)$ be a finite-dimensional representation of $G$.
	If $\delta \colon \aB_0 \to \aB_0$ is a~\Star-de\-ri\-va\-tion, then the linear map
	\begin{gather*}
			\nabla_\delta^\sigma \colon \ \Gamma_{\aA_0}(\sigma) \to \Gamma_{\aA_0}(\sigma), \qquad \nabla_\delta^\sigma(x) := \sum_{k=1}^d s_k \acts \delta \left( \rprod{\aB_0}{s_k,x} \right)
	\end{gather*}
satisfies the following equations
	\begin{gather}
		\label{eq:Leibniz}
		\nabla_\delta^\sigma(x \acts b) = \nabla_\delta^\sigma(x) \acts b + x \acts \delta(b),
		\\
		\label{eq:commutator}
		\delta \left( \rprod{\aB_0}{x,y} \right) = \rprod{\aB_0}{\nabla_\delta^\sigma(x),y} + \rprod{\aB_0}{x,\nabla_\delta^\sigma(y)}
	\end{gather}
for all $x,y \in \Gamma_{\aA_0}(\sigma)$ and $b\in \aB_0$.
\end{Lemma}
\begin{proof}
	Let $x,y \in \Gamma_{\aA_0}(\sigma)$ and $b\in \aB_0$.
	Using first the right $\aB_0$-linearity of $\rprod{\aB_0}{\cdot,\cdot}$, second the derivation property of $\delta$, and finally the reproducing formula for $x$, we obtain
	\begin{align*}
		\nabla_\delta^\sigma(x \acts b)
		&=
		\sum_{k=1}^d s_k \acts \delta \left( \rprod{\aB_0}{s_k,x \acts b} \right)
		=
		\sum_{k=1}^d s_k \acts \delta \left( \rprod{\aB_0}{s_k,x} \cdot b \right)
		\\
		&= \sum_{k=1}^d s_k \acts \left( \delta \left( \rprod{\aB_0}{s_k,x} \right) \cdot b
		+
		\rprod{\aB_0}{s_k,x} \cdot \delta(b) \right)
		=
		\nabla_\delta^\sigma(x) \acts b + x \acts \delta(b).
	\end{align*}
	Likewise, it may be concluded that
	\begin{gather*}
		\rprod{\aB_0}{\nabla_\delta^\sigma(x),y}
		+
		\rprod{\aB_0}{x,\nabla_\delta^\sigma(y)}
		=
		\sum_{k=1}^d \left( \rprod{\aB_0}{s_k \acts \delta \left( \rprod{\aB_0}{s_k,x} \right),y}
		+
		\rprod{\aB_0}{x,s_k \acts \delta \left( \rprod{\aB_0}{s_k,y} \right)} \right)
		\\
		\quad{} =
		 \sum_{k=1}^d \left( \delta \left( \rprod{\aB_0}{x,s_k} \right) \cdot \rprod{\aB_0}{s_k,y}
		 + \rprod{\aB_0}{x,s_k} \cdot \delta \left( \rprod{\aB_0}{s_k,y} \right) \right)
		 =
		 \sum_{k=1}^d \delta \left( \rprod{\aB_0}{x,s_k} \cdot \rprod{\aB_0}{s_k,y} \right)
		 \\
		 \quad{} = \sum_{k=1}^d \delta \left( \rprod{\aB_0}{x,s_k \acts \rprod{\aB_0}{s_k,y}} \right)
		 =
		 \delta \left( \rprod{\aB_0}{x, \sum_{k=1}^d s_k \acts \rprod{\aB_0}{s_k,y}} \right)
		 =
		 \delta \left( \rprod{\aB_0}{x,y} \right),
	\end{gather*}
	where for the second equality we have exploited the fact that $\delta$ is a \Star-derivation.
\end{proof}

We refer to equation~\eqref{eq:Leibniz} as the Leibniz rule and to equation~\eqref{eq:commutator} as the metric compatibility between the inner product and the map $\nabla_\delta^\sigma$.
In addition, we draw attention to the fact that Lemma~\ref{lem:conn} entails that the map
\begin{gather*}
	\nabla^\sigma \colon \ \der(\aB_0) \times \Gamma_{\aA_0}(\sigma) \to \Gamma_{\aA_0}(\sigma), \qquad \nabla^\sigma(\delta,x) := \nabla_\delta^\sigma(x)
\end{gather*}
is a so-called metric connection over $\Gamma_{\aA_0}(\sigma)$ (see, e.g.,~\cite{Jo17,DuVio88}).

\begin{Corollary}
	For each finite-dimensional representation $(\sigma,V_\sigma)$ of $G$ the associated vector bundle $\Gamma_{\aA_0}(\sigma)$ admits a metric connection.
\end{Corollary}

Next, let $\chi\colon \der(\aB_0)_{(\gamma,\omega)} \to \der_G(\aA_0)$ be a connection of $\aA_0$.
It is straightforward to check that, for each $\delta \in \der(\aB_0)_{(\gamma,\omega)}$, the map
\begin{gather*}
	(\nabla^\chi)^\sigma_\delta \colon \ \Gamma_{\aA_0}(\sigma) \to \Gamma_{\aA_0}(\sigma),
	\qquad
	(\nabla^\chi)^\sigma_\delta(x) := \chi(\delta) \tensor \id_{V_\sigma}(x)
\end{gather*}
is well-defined, linear, and satisfies the Leibniz rule.
Moreover, since
\begin{gather*}
	(\nabla^\chi)^\sigma_\delta(x)
	=
	\sum_{k=1}^d \chi(\delta) \tensor \id_{V_\sigma}(s_k)\acts \rprod{\aB_0}{s_k,x} + \nabla^\sigma_\delta(x) \qquad \forall x \in \Gamma_{\aA_0}(\sigma),
\end{gather*}
similar computations as in the proof of Lemma~\ref{lem:conn} yield the following result:

\begin{Corollary}
	Let $\chi\colon \der(\aB_0)_{(\gamma,\omega)} \to \der_G(\aA_0)$ be a connection of $\aA_0$.
Furthermore, let $(\sigma,V_\sigma)$ be a finite-dimensional representation of $G$.
	Then
	\begin{gather*}
\nabla^{\chi,\sigma}\colon \ \der(\aB_0)_{(\gamma,\omega)} \times \Gamma_{\aA_0}(\sigma) \to \Gamma_{\aA_0}(\sigma),
		\qquad
\nabla^{\chi,\sigma}(\delta,x) := (\nabla^\chi)^\sigma_\delta(x)
\end{gather*}
	is a metric connection if and only if
	\begin{gather*}
		\rprod{\aB_0}{\chi(\delta) \tensor \id_{V_\sigma}(x),y} + \rprod{\aB_0}{x,\chi(\delta) \tensor \id_{V_\sigma}(y)} = 0
	\end{gather*}
	for all $\delta \in \der(\aB_0)_{(\gamma,\omega)}$ and $x,y \in \Gamma_{\aA_0}(\sigma)$.
\end{Corollary}

To proceed, we bring to mind that, given a finitely generated projective right $\aB_0$-module $E$ together with a connection $\nabla\colon \der(\aB_0) \times E \to E$, its \emph{curvature} $R := R_\nabla$ with respect to $\nabla$ is the map defined by
	\begin{gather*}
		R \colon \ \der(\aB_0) \times \der(\aB_0) \times E \to E, \qquad R(\delta_1,\delta_2,e) := [ \nabla_{\delta_1},\nabla_{\delta_2} ] (e) - \nabla_{ [\delta_1,\delta_2] } (e).
	\end{gather*}

\begin{Lemma}\label{lem:cur}
	Let $(\sigma,V_\sigma)$ be a finite-dimensional representation of $G$.
	Then the curvature $R^\sigma := R_{\nabla^\sigma}$ of $\Gamma_{\aA_0}(\sigma)$ with respect to $\nabla^\sigma$ takes the form
	\begin{gather*}
	 	R^\sigma(\delta_1,\delta_2,x) = \sum_{k,l=1}^d s_k \acts \bigl( \delta_1 ( \rprod{\aB_0}{s_k,s_l} ) \cdot \delta_2 ( \rprod{\aB_0}{s_l,x} ) - s_k \acts \delta_2 ( \rprod{\aB_0}{s_k,s_l} ) \cdot \delta_1 ( \rprod{\aB_0}{s_l,x} ) \bigr)
	\end{gather*}
	for all $\delta_1,\delta_2 \in \der(\aB_0)$ and $x \in \Gamma_{\aA_0}(\sigma)$.
\end{Lemma}
\begin{proof}
	Let $\delta_1,\delta_2 \in \der(\aB_0)$ and $x \in \Gamma_{\aA_0}(\sigma)$. Then a straightforward computation gives
	\begin{gather*}
 R^\sigma(\delta_1,\delta_2,x) = [ \nabla_{\delta_1}^\sigma,\nabla_{\delta_2}^\sigma ] (x) - \nabla_{ [\delta_1,\delta_2] }^\sigma (x) = \nabla_{\delta_1}^\sigma ( \nabla_{\delta_2}^\sigma(x) ) - \nabla_{\delta_2}^\sigma ( \nabla_{\delta_1}^\sigma(x) ) - \nabla_{ [\delta_1,\delta_2] }^\sigma (x)
		\\
\hphantom{R^\sigma(\delta_1,\delta_2,x)}{} = \sum_{k,l=1}^d s_k \acts \big( \delta_1 ( \rprod{\aB_0}{s_k,s_l} \cdot \delta_2 ( \rprod{\aB_0}{s_l,x} ) ) - \delta_2 ( \rprod{\aB_0}{s_k,s_l} \cdot \delta_1 ( \rprod{\aB_0}{s_l,x} ) ) \big)
		\\
\hphantom{R^\sigma(\delta_1,\delta_2,x)=}{} - \sum_{k=1}^d s_k \acts [\delta_1,\delta_2] ( \rprod{\aB_0}{s_k,x})
		\\
\hphantom{R^\sigma(\delta_1,\delta_2,x)}{}
 = \sum_{k,l=1}^d s_k \acts \bigl( \delta_1 ( \rprod{\aB_0}{s_k,s_l} ) \cdot \delta_2 ( \rprod{\aB_0}{s_l,x}) + \rprod{\aB_0}{s_k,s_l} \cdot \delta_1( \delta_2 ( \rprod{\aB_0}{s_l,x} ))
		\\			
\hphantom{R^\sigma(\delta_1,\delta_2,x)=}{} - \delta_2 ( \rprod{\aB_0}{s_k,s_l} ) \cdot \delta_1 ( \rprod{\aB_0}{s_l,x} ) - \rprod{\aB_0}{s_k,s_l} \cdot \delta_2 ( \delta_1 ( \rprod{\aB_0}{s_l,x} ) ) \bigr)
		\\
\hphantom{R^\sigma(\delta_1,\delta_2,x)=}{} - \sum_{l=1}^d s_l \acts ( \delta_1 ( \delta_2 ( \rprod{\aB_0}{s_l,x} ) ) - \delta_2 ( \delta_1 ( \rprod{\aB_0}{s_l,x} ) ) )
		\\
 \hphantom{R^\sigma(\delta_1,\delta_2,x)}{}
 = \sum_{k,l=1}^d s_k \acts \bigl( \delta_1 ( \rprod{\aB_0}{s_k,s_l} ) \cdot \delta_2 ( \rprod{\aB_0}{s_l,x} ) - s_k \acts \delta_2 ( \rprod{\aB_0}{s_k,s_l} ) \cdot \delta_1 ( \rprod{\aB_0}{s_l,x} ) \bigr).
		\!\!\!\!\!\tag*{\qed}
	\end{gather*}
	\renewcommand{\qed}{}
\end{proof}

As a corollary we can conclude that cleft actions have vanishing curvature:

\begin{Corollary}\label{cor:cur0}
Let $(\aA,G,\alpha)$ be cleft. For a finite-dimensional representation $(\sigma,V_\sigma)$ of~$G$ let~$\nabla^\sigma$ denote the metric connection. Then $R^\sigma(\delta_1,\delta_2,x) = 0$ for all $\delta_1,\delta_2 \in \der(\aB_0)$ and $x \in \Gamma_{\aA_0}(\sigma)$, i.e., the curvature vanishes identically.
\end{Corollary}

\begin{Remark} With a little more effort one can also find a similar formula for the curvature associated with the metric connection $\nabla^{\chi,\sigma}$.
\end{Remark}

\begin{Remark}
	Corollary~\ref{cor:cur0} implies that cleft C\Star-dynamical systems yield trivial (primary) characteristic classes.
\end{Remark}

\subsection*{Acknowledgement}

We gratefully acknowledge the Centre International de Rencontres Math\'ematiques, \href{https://conferences.cirm-math.fr/2177}{REB}~2177, as well as Blekinge Tekniska H\"ogskola for supporting this research.
The first name author expresses his gratitude to iteratec GmbH.
Last but not least, we wish to thank the anonymous referees for providing fruitful criticism that helped to improve the manuscript.

\pdfbookmark[1]{References}{ref}
\LastPageEnding

\end{document}